\documentclass[12pt,reqno]{amsart}
\usepackage{latexsym,amsmath,mathtools,amsfonts,amssymb,amsthm,mathrsfs,color}

\usepackage[T1]{fontenc}
\usepackage{fourier}
\usepackage{xcolor}
\usepackage{fullpage}

\usepackage{tikz}
\usepackage{pgfplots}
\pgfplotsset{compat=1.15}
\usetikzlibrary{arrows}
\usetikzlibrary[patterns]
\usepackage[latin1]{inputenc}
\usepackage{anysize}
\usepackage{graphicx}
\usepackage{url}
\usepackage[noadjust]{cite}

\usepackage{comment}
\newcommand{\bb}[1]{\mathbb{#1}}
\newcommand{\cc}[1]{\mathcal{#1}}

\theoremstyle{definition}
\newtheorem{theorem}{Theorem}[section]
\newtheorem{question}[theorem]{Question}

\newtheorem{remark}[theorem]{Remark}

\newtheorem{lemma}[theorem]{Lemma}
\newtheorem{claim}[theorem]{Claim}
\newtheorem{definition}[theorem]{Definition}

\usepackage[shortlabels]{enumitem}

\title{Expansion in supercritical random subgraphs of the hypercube and its consequences}
	\date{}

	\author{Joshua Erde$^{*}$, Mihyun Kang$^{*}$ and Michael Krivelevich$^{\ddagger}$ \\ \\
		\today}
	\thanks{$^{*}$ 
		Institute of Discrete Mathematics, 
		Graz University of Technology, 
		Steyrergasse 30,
		8010 Graz,
		Austria,  
		{\tt \{erde,kang\}@math.tugraz.at}.
		Supported by Austrian Science Fund (FWF): I3747\phantom{}}
	\thanks{$^{\ddagger}$ 
		School of Mathematical Sciences, 
		Sackler Faculty of Exact Sciences, 
		Tel Aviv University,  
		Tel Aviv 6997801,
		Israel, 
		{\tt  krivelev@tauex.tau.ac.il}.
		Supported in part by USA-Israel BSF grant 2018267, and by ISF grant 1261/17.}
\begin{document}
	
	\begin{abstract}
It is well-known that the behaviour of a random subgraph of a $d$-dimensional hypercube, where we include each edge independently with probability $p$, undergoes a phase transition when $p$ is around $\frac{1}{d}$. More precisely, standard arguments show that just below this value of $p$ all components of this graph have order $O(d)$ with probability tending to one as $d \to \infty$ (whp for short), whereas Ajtai, Koml\'{o}s and Szemer\'{e}di [Largest random component of a $k$-cube, Combinatorica 2 (1982), no. 1, 1--7; MR0671140] showed that just above this value, in the \emph{supercritical regime}, whp there is a unique `giant' component of order $\Theta\left(2^d\right)$. We show that whp the vertex-expansion of the giant component is inverse polynomial in $d$. As a consequence we obtain polynomial in $d$ bounds on the diameter of the giant component and the mixing time of the lazy random walk on the giant component, answering questions of Bollob\'{a}s, Kohayakawa and {\L}uczak [On the diameter and radius of random subgraphs of the cube, Random Structures and Algorithms 5 (1994), no. 5, 627--648; MR1300592] and of Pete [A note on percolation on $\mathbb{Z}^d$: isoperimetric profile via exponential cluster repulsion, Electron. Commun. Probab. 13 (2008), 377--392; MR2415145]. Furthermore, our results imply lower bounds on the circumference and Hadwiger number of a random subgraph of the hypercube in this regime of $p$ which are tight up to polynomial factors in $d$.
	\end{abstract}

	\maketitle
	
\section{Introduction}
Percolation is a mathematical process, initially studied by Broadbent and Hammersley \cite{BH57} to model the flow of a fluid through a porous medium whose channels may be randomly blocked. The underlying mathematical model is simple: given a graph $G$, usually some sort of lattice-like graph, the \emph{percolated subgraph} $G_p$ is the random subgraph of $G$ obtained by retaining each edge of $G$ independently with probability $p$. For a more detailed introduction to percolation theory see, e.g., \cite{K82,G89,BR06}.

In this paper we are concerned with percolation on the hypercube. The \emph{$d$-dimensional hypercube} $Q^d$ is the graph with vertex set $V\left(Q^d\right) = \{0,1\}^d$ and in which two vertices are adjacent if they differ in exactly one coordinate. Throughout this paper we will write $n := 2^d$ for the order of the hypercube and we note that $|E\left(Q^d\right)| = \frac{nd}{2}$. The hypercube is a ubiquitous object in graph theory and combinatorics, arising naturally in many contexts, in particular due to its interpretation as the Hasse diagram of the partial order on $[d]:=\{1,2,\ldots,d\}$ given by the subset relation.
	
The random subgraph $Q^d_p$ of the hypercube was first studied by Sapo\v{z}enko \cite{S67} and by Burtin \cite{B67}, who showed that $Q^d_p$ has a threshold for connectivity at $p=\frac{1}{2}$; for a fixed constant $p < \frac{1}{2}$, whp\footnote{Throughout the paper, all asymptotics will be considered as $d \to \infty$ and so, in particular, whp (with high probability) means with probability tending to one as $d \to \infty$.} $Q^d_p$ is disconnected, whereas for $p >\frac{1}{2}$, whp $Q^d_p$ is connected. This result was strengthened by Erd\H{o}s and Spencer \cite{ES79} and by Bollob\'{a}s \cite{B83}, who determined the probability of connectivity of $Q^d_p$ when $p$ is close to $\frac{1}{2}$. Bollob\'{a}s \cite{B90} also showed that $p=\frac{1}{2}$ is the threshold for the existence of a perfect matching in $Q^d_p$. Very recently, answering a longstanding open problem, Condon, Espuny D{\'\i}az, Gir{\~a}o, K{\"u}hn and Osthus \cite{CDGKO21} showed that $p=\frac{1}{2}$ is also the threshold for the existence of a Hamilton cycle in $Q^d_p$.

Motivated by results from the \emph{binomial random graph model}, it was conjectured by Erd\H{o}s and Spencer \cite{ES79} that the component structure of $Q^d_p$ should undergo a \emph{phase transition} at $p=\frac{1}{d}$: it is relatively easy to see, by a coupling with a branching process, that when $p = \frac{1-\epsilon}{d}$ for $\epsilon>0$, whp all components of $Q^d_p$ have order $O(d)$, but they conjectured that when $p= \frac{1+\epsilon}{d}$, whp $Q^d_p$ contains a unique `giant' component $L_1\left(Q^d\right)$, whose order is linear in $n=2^d$. This conjecture was confirmed by Ajtai, Koml\'{o}s and Szemer\'{e}di \cite{AKS81}. 

\begin{theorem}[Ajtai, Koml\'{o}s and Szemer\'{e}di \cite{AKS81}]\label{t:AKSsimple}
Let $\epsilon > 0$ and let $p =\frac{1+\epsilon}{d}$. Then there is a constant $\gamma > 0$ such that whp $Q^d_p$ contains a component of order at least $\gamma n$.
\end{theorem}

These results were later extended to a wider range of $p$, describing more precisely the component structure of $Q^d_p$ when $p = \frac{1+\epsilon}{d}$ with $\epsilon=o(1)$ by Bollob\'{a}s, Kohayakawa and {\L}uczak \cite{BKL92}, by Borgs, Chayes, van der Hofstad, Slade and Spencer \cite{BCVSS06} and by Hulshof and Nachmias \cite{HN20}, with the correct width of the critical window in this model being only recently identified by van der Hofstad and Nachmias \cite{HN17}. McDiarmid, Scott and Withers \cite{MSW21} also give a description of the component structure of $Q^d_p$ for fixed $p \in \left(0, \frac{1}{2}\right)$, when $p$ is quite far from the critical window, but still below the connectivity threshold. For a more detailed background on the phase transition in this model, see the survey of van der Hofstad and Nachmias \cite{HN14}.

In this paper we are interested in the typical structural properties of the giant component $L_1$ of $Q^d_p$ in the \emph{supercritical regime}, where $p=\frac{1+\epsilon}{d}$ for some fixed $\epsilon > 0$. Our main result, from which we will be able to deduce a lot of structural information about $L_1$, concerns its \emph{expansion} properties, in particular vertex-expansion. Informally, a graph has good vertex-expansion if all sufficiently small vertex sets have a large vertex boundary, expressing a kind of discrete isoperimetric inequality. The notion of graph expansion seems to be quite a fundamental one to the study of graphs, demonstrating a deep link between the geometric and structural properties of a graph, its algebraic spectrum and also the mixing time of the random walk on the graph. For these reasons and more, graph expansion has turned out to have fundamental importance in many diverse areas of discrete mathematics and computer science. For a comprehensive introduction to expander graphs, see the survey of Hoory, Linial and Widgerson \cite{HLW06}. In particular, notions of expansion have turned out to be a powerful tool in the study of random structures. See, for example, the survey paper of Krivelevich \cite{K19}. 

In order to motivate our results and methods, let us discuss briefly what is known in a simpler model of percolation, the binomial random graph. The binomial random graph $G(d+1,p)$, introduced by Gilbert \cite{G59}, is a percolated subgraph of the complete graph $K_{d+1}$, where we retain each edge with probability $p$. As a particularly simple model of percolation, where the underlying graph $G$ lacks the geometric structure of $\bb{Z}^d$ or $Q^d$, the binomial random graph has been extensively studied. A particularly striking feature of this model is the phase transition that it undergoes at $p = \frac{1}{d}$, exhibiting vastly different behaviour when $p = \frac{1 - \epsilon}{d}$ to when $p = \frac{1+\epsilon}{d}$ (where $\epsilon$ is a positive constant). 

More precisely, it follows from results of Erd\H{o}s and R\'{e}nyi \cite{ER59} that when $p=\frac{1-\epsilon}{d}$ for a fixed $\epsilon>0$, whp every component of $G(d+1,p)$ has order at most $O(\log d)$ and is either a tree or unicyclic, whereas when $p=\frac{1+\epsilon}{d}$, whp $G(d+1,p)$ contains a unique `giant component' $L_1\left(G(d+1,p)\right)$ of order $\Omega(d)$, whose structure is quite complex. We note that this is only a very broad picture of the phase transition in this model, and much more precise results are known, in particular extending these results into the \emph{weakly supercritical regime} where $\epsilon = o(1)$ and $\epsilon^3 d \to \infty$, see, for example, the works of Bollob\'{a}s \cite{B84} and {\L}uczak \cite{L90}. However this is not the focus of our paper. 

Much subsequent work has focused on the structural properties of the giant component $L_1 = L_1(G(d+1,p))$ in the supercritical regime. For example, a well-known result of Ajtai, Koml\'{o}s and Szemer\'{e}di \cite{AKS81a} shows that in this regime whp $L_1$ contains a path of length $\Omega(d)$, from which it is easy to deduce that the \emph{circumference}, the length of the largest cycle, of $L_1$ is of order $\Omega(d)$. Fountoulakis, K{\"u}hn and Osthus \cite{FKO08} showed that whp the \emph{Hadwiger number}, the size of the largest complete minor, of $L_1$ is of order $\Omega(\sqrt{d})$. Chung and Lu \cite{CL01} showed that whp $L_1$ has diameter $O(\log d)$ and this result was later strengthened, in particular determining the correct leading constant, by work of Fernholz and Ramachandran \cite{FR07} and of Riordan and Wormald \cite{RW10}. Benjamini, Kozma and Wormald \cite{BKW14} and Fountoulakis and Reed \cite{FR08} showed that whp the mixing time of the lazy random walk on $L_1$ is $O\left((\log d)^2\right)$ and Berestycki, Lubetzky, Peres and Sly \cite{BLPS18} showed that if we start the lazy random walk from a uniformly chosen vertex of $L_1$, then whp the mixing time is $O(\log d)$ (for definitions related to mixing time see Section \ref{s:mixing}). We note that this is just a small subset of the results known about the supercritical random graph, chosen judiciously for comparison to our results later. For more general background on the theory of random graphs, see \cite{B01,JLR00,FK16}.

A lot of structural information about the giant component of $G(d+1,p)$, including many of the results mentioned above, can be deduced as consequences of its expansion properties. Indeed, it is known that the expansion properties of a graph can be linked to various structural graph properties, for example its diameter and circumference, and furthermore there are well-known links between expansion and mixing times of Markov chains. See Sections \ref{s:prelim} and \ref{s:mixing} for more precise statements.

Given a graph $G$ and a subset $S \subseteq V(G)$ we write $N_G(S)$ for the \emph{external neighbourhood} of $S$ in $G$, that is, the set of vertices in $V(G) \setminus S$ which have a neighbour in $S$ and we write $S^c := V(G) \setminus S$ and $e_G(S,S^c)$ for the number of edges between $S$ and $S^c$ in $G$.  
\begin{definition}
We say a graph $G$ is an \emph{$\alpha$-expander} if $|N_G(S)| \geq \alpha |S|$ for every $S \subseteq V(G)$ such that $|S| \leq \frac{|V(G)|}{2}$, where $\alpha$ is the \emph{expansion ratio}. 
\end{definition}
Similarly we say a graph $G$ is an \emph{$\alpha$-edge-expander} if $e_G(S,S^c) \geq \alpha |S|$ for all $S$ such that $\sum_{v \in S} d_G(v) \leq |E(G)|$.

It can easily be seen that $G(d+1,p)$ is whp not an expander in the supercritical regime, since the graph is likely not connected. In fact, standard results imply that even the \emph{$2$-core}, the largest subgraph of minimum degree at least two, of the giant component of $G(d+1,p)$ does not have constant vertex- or edge-expansion, since it typically contains logarithmically long \emph{bare paths}, paths in which each internal vertex has degree two. However, it was shown by Benjamini, Kozma and Wormald \cite{BKW14} that in the supercritical regime whp the giant component of $G(d+1,p)$ is a \emph{decorated expander}, which roughly means that it has a linear sized subgraph which is an $\alpha$-edge-expander for some constant $\alpha >0$, and the deletion of this subgraph splits the giant component into logarithmically small pieces. It is not hard to show that in the supercritical regime the existence of a linear sized subgraph which is an $\alpha$-edge-expander implies the existence of one which is an $\alpha'$-expander for some $\alpha' >0$, the existence of which was also shown by Krivelevich \cite{K18} using general considerations of expansion in locally sparse graphs. We also note that the work of Ding, Lubetzky and Peres \cite{DLP14} gives a particularly simple model contiguous to the giant component in the supercritical regime, which implies that whp the \emph{kernel} of the giant component, the graph obtained by contracting all bare paths in the $2$-core, is an $\alpha$-expander for some fixed $\alpha > 0$, and from which it is possible to determine the likely expansion properties of the giant component and its $2$-core. Then, using these connections between the structural properties of a graph and its expansion mentioned previously, many properties of the giant component of $G(d+1,p)$ can be deduced as consequences of the results of Benjamini, Kozma and Wormald \cite{BKW14}, of Krivelevich \cite{K18} and of Ding, Kim, Lubetzky and Peres \cite{DLP14}.

More recently, random subgraphs $G_p$ of an arbitrary graph $G$ of large minimum degree $\delta(G) \geq d$ have been studied. It has been observed that  some of the complex behaviour which occurs whp in $G(d+1,p)$ once we pass the critical point of $p=\frac{1}{d}$ also occurs whp in $G_p$ in the same regime of $p$. For example, when $p= \frac{1+\epsilon}{d}$ for $\epsilon > 0$, it has been shown that whp $G_p$ contains a path or cycle of length at least linear in $d$, see \cite{KS13,KS14,EJ18}. Furthermore, for this range of $p$ it has been shown by Frieze and Krivelevich \cite{FK13} that whp $G_p$ is non-planar and by Erde, Kang and Krivelevich \cite{EKK20} that in fact, whp $G_p$ has Hadwiger number $\tilde{\Omega}\left(\sqrt{d}\right)$.\footnote{The notation $\tilde{\Omega}\left(\cdot\right)$ here is hiding a polylogarithmic factor in $d$.} 

Whilst the model $G(d+1,p)$ shows that these results are optimal when $G$ can be an arbitrary graph, for specific graphs, and in particular for $G = Q^d$, they may be far from the truth. Indeed, it is plausible that circumference and Hadwiger number of $Q^d_p$ could be exponentially large in $d$ in the supercritical regime. Furthermore, since the host graph $G$ can be chosen arbitrarily, we cannot hope to prove much about global properties of $Q^d_p$, such as the diameter or mixing time, by considering the far more general model of arbitrary random subgraphs $G_p$.

The main aim of this paper is to show that whp the giant component of $Q^d_p$ in the supercritical regime has good expansion properties. Given constants $\alpha,\beta >0$ and a statement $A$, we will write `Let $\alpha \ll \beta$. Then $A$ holds' to indicate that there is some fixed, implicit function $f$ such that $A$ holds for all $\alpha \leq f(\beta)$.

\begin{theorem}\label{t:expansionsimple1}
Let $\epsilon > 0$, let $p=\frac{1+\epsilon}{d}$ and let $L_1$ be the largest component of $Q^d_p$. Then there exists a constant $\beta>0$ such that whp $L_1$ is a $\beta d^{-5}$-expander.
\end{theorem}

Furthermore, we show that whp the giant component of $Q^d_p$ contains an almost spanning subgraph with much better expansion.

\begin{theorem}\label{t:expansionsimple2}
Let $0 < \alpha \ll \epsilon$, let $p=\frac{1+\epsilon}{d}$ and let $L_1$ be the largest component of $Q^d_p$. Then there exist a constant $\beta>0$ and a subgraph $H$ of $L_1$ of size at least $(1-\alpha)|V(L_1)|$ such that whp $H$ is a $\beta d^{-2} (\log d)^{-1}$-expander.
\end{theorem}

We will be able to deduce certain structural consequences from the vertex-expansion of the giant component which are almost optimal, up to polynomial factors in $d$.

The first consequence will be a bound on the mixing time of the lazy random walk on the giant component $L_1$ of $Q^d_p$. Answering a well-known question (see, e.g., Pete \cite{P08} and van der Hofstad and Nachmias \cite{HN17}), we show that whp the mixing time of this random walk is polynomial in $d$.

\begin{theorem}\label{t:mixingtime}
Let $\epsilon >0$, let $p=\frac{1+\epsilon}{d}$ and let $L_1$ be the largest component in $Q^d_p$. Then whp the mixing time of the lazy random walk on $L_1$ is $O\left( d^{11}\right)$.
\end{theorem}

Next, we consider the diameter of the giant component in $Q^d_p$. Previously, Bollob\'{a}s, Kohayakawa and {\L}uczak \cite{BKL94d} considered the diameter of a random $Q^d$-process. Although their results mainly concern the structure of $Q^d_p$ close to the connectivity threshold, they asked whether the diameter of any component in a typical $Q^d$-process is ever superpolynomial in $d$. 

In fact, since then it has been shown that the diameter of components in the regime close to the critical window can grow even exponentially large in $d$. Hulshof and Nachmias \cite{HN20} show that in the weakly subcritical regime, when $n^{-\frac{1}{3}} \ll \epsilon = o(1)$ and $p=\frac{1-\epsilon}{d}$, whp the maximal diameter of a component in $Q^d_p$ is $(1+o(1)) \epsilon^{-1} \log \left(\epsilon^3 n\right)$, although they mention that this is not achieved by the largest component, which they conjecture to have diameter $\Theta\left(\epsilon^{-1} \sqrt{\log \left(\epsilon^3 n\right)} \right)$. Heydenreich and van der Hofstad \cite{HH11} mention that their methods also show that the diameter of the largest component of the critical percolated hypercube is $\Theta_p\left(n^{\frac{1}{3}}\right)$ and it is also stated by van der Hofstad and Nachmias \cite{HN17} that in the weakly supercritical regime, when $n^{-\frac{1}{3}} \ll \epsilon = o(1)$ and $p=\frac{1+\epsilon}{d}$, whp the diameter of the giant component in $Q^d_p$ is $(1+o(1)) \epsilon^{-1} \log \left(\epsilon^3 n\right)$.

Hence, the question of Bollob\'{a}s, Kohayakawa and {\L}uczak \cite{BKL94d} perhaps only makes sense once we are quite far from the critical window. In this range, we give a polynomial bound on the likely diameter of the giant component.

\begin{theorem}\label{t:diameter}
Let $\epsilon >0$ and let $p=\frac{1+\epsilon}{d}$. Then whp the largest component $L_1$ of $Q^d_p$ has diameter $O\left(d^3\right)$.
\end{theorem}


Finally, we also consider the circumference and the Hadwiger number of $Q^d_p$.

\begin{theorem}\label{t:cycles}
Let $\epsilon >0$ and let $p=\frac{1+\epsilon}{d}$. Then whp the circumference of $Q^d_p$ is 
\[
\Omega\left( nd^{-2} (\log d)^{-1} \right).
\]
\end{theorem}

\begin{theorem}\label{t:minors}
Let $\epsilon >0$ and let $p=\frac{1+\epsilon}{d}$. Then whp the Hadwiger number of $Q^d_p$ is 
\[
\Omega \left( \sqrt{n}d^{-2} (\log d)^{-1} \right).
\]
\end{theorem}

We note that, since whp ${\left|E\left(Q^d_p\right)\right| = O(n)}$ when $p = O\left(\frac{1}{d}\right)$ and the Hadwiger number $h(G)$ of a graph $G$  satisfies $\binom{h(G)}{2} \leq |E(G)|$, it follows that both of these results are optimal up to polynomial terms in $d$. We note further that a bound on the likely circumference of $Q^d_p$ was also obtained in concurrent work by Haslegrave, Hu, Kim, Liu, Luan and Wang \cite{HHKLLW21}, who showed the likely existence of a cycle of length $\Omega\left( nd^{-32} \right)$ using different methods.

The paper is structured as follows. In Section \ref{s:prelim} we collect some lemmas which will be useful in the rest of the paper. Then in Section \ref{s:expansion} we prove our main results (Theorems \ref{t:expansionsimple1} and \ref{t:expansionsimple2}) on the likely expansion of the giant component of $Q^d_p$. In Section \ref{s:consequences} we use our main result to prove Theorems \ref{t:mixingtime}--\ref{t:minors}, and then finally in Section \ref{s:discussion} we mention some open problems and avenues for further research.

\section{Preliminaries}\label{s:prelim}
For real numbers $x,y,z$ we will write $x=y\pm z$ to mean that $y-z \leq x \leq y+z$.

If $G$ is a graph and $X$ is either a subgraph of $G$, or a subset of $V(G)$, then we will write $G-X$ for the induced subgraph of $G$ on $V(G) \setminus V(X)$ or $V(G) \setminus X$, respectively. Given a $k \in \mathbb{N}$ and a subset $S \subseteq V(G)$, we will write $N^k_G(S)$ for the \emph{$k$th external neighbourhood} of $S$ in $G$, that is, the set of vertices in $S^c:=V(G) \setminus S$ which are at distance at most $k$ from $S$ in $G$. When $k=1$, or when the graph is clear from context, we will omit the subscript or superscript, respectively. Given (not necessarily disjoint) subsets $A,B \subseteq V(G)$ we will write $e_G(A,B)$ for the number of edges in $G$ with one endpoint in $A$ and the other in $B$ and we write $e_G(S) = e_G(S,S)$. We say a subset $S \subseteq V(G)$ is \emph{connected (in $G$)} if $G[S]$ is connected.

Throughout the paper, unless the base is explicitly mentioned, all logarithms will be the natural logarithm. We will also omit floor and ceiling signs for ease of presentation.

We will want to use a more explicit form of Theorem \ref{t:AKSsimple}. It is stated in \cite{AKS81} that a careful treatment of their proof gives the following result, which also appears explicitly in the work of Bollob\'{a}s, Kohayakawa and {\L}uczak \cite[Theorem 32]{BKL94}.

\begin{theorem}[\cite{AKS81}]\label{t:AKSfine}
Let $0 < c \ll \delta$, let $p=\frac{1+\delta}{d}$ and let $\gamma:= \gamma(\delta)$ be the survival probability of the ${\text{Po}(1+\delta)}$ branching process. Then whp there is a unique component $L_1$ of order at least $c n$ in $Q^d_p$ and $|V(L_1)| = (\gamma \pm c) n$.
\end{theorem}
	
The following simple lemma, which is a slight adaptation of a result in \cite{KN06}, allows us to decompose a tree into roughly equal sized parts.
	\begin{lemma}\label{l:treedecomp}
	Let $T$ be a tree such that $\Delta(T) \leq C_1$, all but $r$ vertices of $T$ have degree at most $C_2\leq C_1$ and $|V(T)| \geq \ell$, for some $C_1,C_2,\ell,r >0$. Then there exist disjoint vertex sets $A_1,\ldots, A_s \subseteq V(T)$ such that
		\begin{itemize}
			\item $V(T) = \bigcup_{i=1}^s A_i$; 
			\item $T[A_i]$ is connected for each $1 \le i \le s$;
			\item $T[A_i]$ has diameter at most $2\ell$;
			\item $\ell \leq |A_i| \leq  C_1 \ell$ for each $1 \le i \le r$; and
			\item $\ell \leq |A_i| \leq  C_2 \ell$ for each $r < i \le s$.
		\end{itemize}
	\end{lemma}
	\begin{proof}
We choose an arbitrary root $w$ for $T$. For a vertex $v$ in a rooted tree $S$, let us write $S_v$ for the subtree of $S$ rooted at $v$. 

We construct the vertex sets $A_i$ inductively. Let us start by setting $T(0)=T$. Given a tree $T(i)$ rooted at $w$ such that $|V\left(T(i)\right)| \geq \ell$, let $v_i$ be a vertex of maximal distance from $w$ such that $\left|V\left(T(i)_{v_i}\right)\right| \geq \ell$. We take $A_{i+1} = V\left(T(i)_{v_i}\right)$ and let $T(i+1) = T(i) - T(i)_{v_i}$. We stop when $|V\left(T(i)\right)| < \ell$, and in that case we add $V\left(T(i)\right)$ to the final $A_i$. Finally, let us re-order the sets $A_i$ so that they are non-increasing in size.

We claim that the sets $A_1, A_2, \ldots, A_s$ satisfy the conclusion of the lemma. Indeed, the first two properties are clear by construction. Note that each $A_i$ is the union of the vertices of $T(i)_x$ over all children $x$ of $v_i$, together with a connected set of size at most $\ell$ which contains $v_i$, and by our choice of $v_i$, $|T(i)_x| < \ell$ for every child $x$ of $v_i$. 

In particular, it follows that every vertex in $A_i$ is at distance at most $\ell$ from $v_i$, and so $T[A_i]$ has diameter at most $2\ell$, and the third property holds. Furthermore, if $v_i \neq w$, then $v_i$ has $d(v_i)-1$ children and so, it follows that $|A_{i+1}| \leq (d(v_i)-1)(\ell-1) +\ell \leq d(v_i) \ell$. Similarly, if ${v_i=w}$, then we note that $A_{i+1}=V(T(i))$ and so $|A_{i+1}| \leq d(w)(\ell-1) + 1 \leq d(w) \ell$. Therefore, since all but $r$ vertices of $T$ have degree at most $C_2$, the fourth and fifth properties also hold.
	\end{proof}
	
We will also need the following bound on the number of subtrees of a graph.

\begin{lemma}[{\cite[Lemma 2]{BFM98}}]\label{l:treecount}
Let $G$ be a graph with maximum degree $\Delta$, let $v \in V(G)$ and let $t(v,k)$ be the number of rooted trees in $G$ which have root $v$ and $k$ vertices. Then
\[
t(v,k) \leq \frac{k^{k-2}\Delta^{k-1}}{(k-1)!} \leq (e\Delta)^{k-1}.
\]
\end{lemma}
	
The following theorem allows us to deduce the existence of a long cycle from vertex-expansion properties of a graph. For wider context on properties of expanding graphs, see the survey of Krivelevich \cite{K19}.

\begin{theorem}[{\cite[Theorem 1]{K19a}}]\label{t:cycleexpander}
Let $k \geq 1, t \geq 2$ be integers. Let $G$ be a graph on more than $k$ vertices satisfying
\[
\left|N(W)\right| \geq t, \qquad \text{for every } W \subseteq V(G) \text{ with } \frac{k}{2} \leq |W| \leq k.
\]
Then $G$ contains a cycle of length at least $t+1$.
\end{theorem}

Furthermore, the next theorem allows us to deduce the existence of a large complete minor in a graph without any small separators. It is easy to see that graphs with good vertex-expansion do not contain any small separators, and in fact it is known (see \cite[Section 5]{K19}) that the converse is true, in the sense that graphs without small separators must contain large induced subgraphs with good vertex-expansion.

\begin{theorem}[{\cite[Theorem 1.2]{KR10}}]\label{t:minorexpander}
Let $G$ be a graph with $N$ vertices and with no $K_t$-minor. Then $V(G)$ contains a subset $S$ of size $O\left(t \sqrt{N}\right)$ such that each connected component of $G - S$ has at most $\frac{2}{3}N$ vertices.
\end{theorem}

Given a discrete random variable $X$ taking values in $\mathcal{X}$ the \emph{entropy} of $X$ is given by
\[
H(X) := \sum_{x \in \mathcal{X}} - p(x) \log_2 (p(x)),
\]
where $p(x) = \mathbb{P}(X=x)$. We will need only two basic facts about the entropy function.
\begin{lemma}\label{l:entropy}
\leavevmode
\begin{enumerate}[(i)]
\item\label{i:unifent} $H(X) \leq \log_2 (|\mathcal{X}|)$ with equality iff $X$ is uniformly distributed,
\item\label{i:jointent} $H(X_1,X_2, \ldots, X_d) \leq \sum_{i=1}^d H(X_i)$,
\end{enumerate}
where the \emph{joint entropy} $H(X_1,X_2, \ldots, X_d)$ is the entropy of the random vector $(X_1,X_2, \ldots, X_d)$. 
\end{lemma}
For proofs of these facts, and more background on discrete entropy, see, e.g., \cite[Chapter 15]{AS16}.

We will use the following Chernoff type bounds on the tail probabilities of the binomial distribution, see, e.g., \cite[Appendix A]{AS16}.
\begin{lemma}\label{l:Chernoff}
Let $N \in \mathbb{N}$, let $p \in [0,1]$ and let $X \sim \text{Bin}(N,p)$.

\begin{enumerate}[(i)]
\item\label{i:chernoff1} For every positive $a$ with $a \leq \frac{Np}{2}$,
\[
\mathbb{P}\left(\left|X -Np \right| > a\right) < 2 \exp\left(-\frac{a^2}{4Np} \right).
\]
\item\label{i:chernoff2} For every positive $b$,
\[
\mathbb{P}\left(X > bNp \right) \leq \left(\frac{e}{b}\right)^{b Np}.
\]
\end{enumerate}
\end{lemma}

In particular, the following two simple consequences of Lemma \ref{l:Chernoff} in our setting will be useful. The first bounds the number of high degree vertices in $Q^d_p$ for small $p$.

\begin{lemma}\label{l:degrees}
Let $c > 0$ be a constant and let $p=\frac{c}{d}$. Then whp $Q^d_p$ contains at most $nd^{-4}$ vertices of degree at least $\log d$.
\end{lemma}
\begin{proof}
For any fixed vertex $v \in V\left(Q^d\right)$, the degree of $v$ in $Q^d_{p}$ is distributed as Bin$(d,p)$, and so by Lemma \ref{l:Chernoff} \ref{i:chernoff2} we have that
\[
\mathbb{P}\left( d_{Q^d_p}(v) \geq \log d  \right) \leq \left(\frac{e(1+c)}{ \log d } \right)^{ \log d } \leq d^{-\frac{\log \log d}{2}}.
\]

It follows that the expected number of vertices in $Q^d_p$ with degree at least $\log d$ is at most $nd^{-\frac{\log \log d}{2}}$. Hence, by Markov's inequality, whp there at most $nd^{-4}$ vertices with degree at least $\log d$.
\end{proof}

We note that the above argument is suboptimal and with a little more care, the bound on the degree of the exceptional vertices could be improved from $\log d$ to $\frac{C \log d}{\log \log d}$ for some suitably large constant $C$. However, for ease of presentation we have not attempted to optimise any logarithmic factors in our proofs.

The second consequence of Lemma \ref{l:Chernoff} allows us to find large matchings in random subsets of edges in $Q^d_p$.

\begin{lemma}\label{l:matching}
Let $\delta > 0$ be a constant, let $p=\frac{\delta}{d}$ and let $F \subseteq E(Q^d)$ be such that $|F| \geq t$. Then there exists a constant $c >0$ such that $F_p$ contains a matching of size at least $c t d^{-1}$ with probability at least $1- \exp\left( -c t d^{-1}\right)$.
\end{lemma}
\begin{proof}
We note that we can assume that $|F|=t$. Let us consider the number of maximal matchings in $F_p$ of size $\ell$.

There are clearly at most $\binom{|F|}{\ell}$ potential maximal matchings of size $\ell$, and given a matching $M$ of size $\ell$ in $F$, in order for it to be a maximal matching in $F_p$ its edges have to appear in $F_p$, which happens with probability $p^{\ell}$, and also there can be no other edges in $F_p$ which are disjoint from $M$. Since there are at most $2 \ell d$ edges which share a vertex with edges in $M$, there is a set of $|F| - 2 \ell d$ edges which do not appear in $F_p$, which happens with probability at most $(1-p)^{|F| - 2 \ell d}$. Hence, by the union bound, the probability that $F_p$ contains a maximal matching of size $\ell$ is at most
\[
\binom{|F|}{\ell} \left(\frac{\delta}{d}\right)^{\ell} \left( 1- \frac{\delta}{d} \right)^{|F| - 2 \ell d}.
\]

In particular, as long as $c \ll \delta$, we can bound the probability $q$ that $F_p$ contains a maximal matching of size $\ell \leq c t d^{-1}$ from above by
\begin{align*}
q \leq \sum_{\ell=1}^{c t d^{-1}} \binom{t}{\ell} \left(\frac{\delta}{d}\right)^{\ell} \left( 1- \frac{\delta}{d} \right)^{t- 2 \ell d} &\leq \sum_{\ell=1}^{c t d^{-1}} \left(\frac{et}{\ell}\right)^{\ell}\left(\frac{\delta}{d}\right)^{\ell} \left( 1- \frac{\delta}{d} \right)^{\frac{t}{2}} \leq \exp \left(- \frac{\delta t}{2d}\right) \sum_{\ell=1}^{c t d^{-1}}\left(\frac{e\delta t}{\ell d}\right)^{\ell},
\end{align*}
since in this range of $\ell$ we have that $t - 2\ell d \geq \frac{t}{2}$. However, since $c \ll \delta$, it can be seen that the ratio of the consecutive terms $\left(\frac{e\delta t}{\ell d}\right)^{\ell}$ is at most $\frac{1}{2}$, and so the sum is dominated by the final term. Hence we can bound
\begin{align*}
q &\leq 2 \exp \left(- \frac{\delta t}{2d}\right) \left(\frac{e\delta}{c}\right)^{c t d^{-1}} \leq  2 \exp\left(\frac{c t}{d} \log\left( \frac{e \delta}{c}\right) - \frac{\delta t}{2d} \right) \leq 2 \exp\left( -ct d^{-1}\right).
\end{align*}
\end{proof}

We note that the conclusion of Lemma \ref{l:matching} is optimal up to a constant factor. Indeed, whp, for example by Lemma \ref{l:Chernoff}, there will only be $O\left(td^{-1}\right)$ edges in $F_p$.

We will use the following well-known result on edge-isoperimetry in the hypercube, originally due to Harper \cite{H64}, see also Lindsey \cite{L64}, Bernstein \cite{B67}, and Hart \cite{H76}.

\begin{theorem}[\cite{H64,L64,B67,H76}]\label{t:iso}
For any $A \subseteq V\left(Q^d\right)$ with $|A| \leq 2^{d-1}$,
\[
e(A,A^c) \geq |A|(d - \log_2 |A|).
\]
\end{theorem}

Finally we will use use the following lemma which bounds the likely number of edges spanned by connected subsets in $Q^d_p$.
\begin{lemma}\label{l:excess}
Let $\delta > 0$ be a constant and let $p=\frac{\delta}{d}$. Then there exists a constant $C:=C(\delta)$ such that whp every subset $S \subseteq V(Q^d)$ such that $|S| \geq d$ and $Q^d_p[S]$ is connected satisfies $e_{Q^d_p}(S) \leq C|S|$.
\end{lemma}
\begin{proof}
Note that, if $S\subseteq V(Q^d)$ has size $|S|=:k$, then, since $Q^d$ is $d$-regular, it follows from Theorem \ref{t:iso} that $e_{Q^d}(S) \leq \frac{k \log_2 k}{2}$. 

Let us bound from above the probability that there exists a subset of $V(Q^d_p)$ of size $k$ which is connected and spans at least $Ck$ many edges. Such a subset must span a tree, which we can specify by choosing a vertex and one of the at most $(ed)^{k-1}$ trees of size $k$ containing that vertex, using Lemma \ref{l:treecount} to bound this quantity. This tree is contained in $Q^d_p$ with probability $p^{k-1}$.

If we let $S$ be the vertex set of this tree, then by the above comment $e_{Q^d}(S) \leq \frac{k \log_2 k}{2}$. In order for $e_{Q^d_p}(S) \geq Ck$ there must be a set of $Ck - (k-1) \geq (C-1)k$ further edges of $Q^d$ in $S$ which appear in $Q^d_p$, which happens with probability at most $p^{(C-1)k}$. 

Hence, writing $C' = C-1$ for ease of presentaiton, by the union bound, the probability that such a set of size $k \geq d$ exists is at most
\begin{align*}
\sum_{k=d}^{n} n(ed)^{k-1} p^{k-1}\binom{\frac{k \log_2 k}{2}}{C'k} p^{C'k} &\leq \sum_{k=d}^{n} 2^k ( e\delta)^{k-1} \left(\frac{e\delta \log_2 k}{2C'd} \right)^{C'k}\\
&\leq \sum_{k=d}^{n}  \left(\frac{2^{\frac{1}{C'}}(e\delta)^{1 + \frac{k-1}{kC'}}}{4C'}\right)^{C'k} =o(1),
\end{align*}
as long as $C' =C-1$ is sufficiently large in terms of $\delta$.

\end{proof}

\section{Expansion in the giant component}\label{s:expansion}
We begin by establishing some likely properties of the giant component of $Q^d_p$ which will be useful in our proof.

The first says that whp the second largest component of $Q^d_p$ in the supercritical regime is only of linear size in $d$. We note that it is mentioned already in \cite{AKS81} that such a result can be shown using methods of Koml\'{o}s, Sulyok and Szemer\'{e}di from \cite{KSS80}, however a proof can be found in \cite[Theorem 31]{BKL92}.

\begin{lemma}[{\cite[Theorem 31]{BKL92}}]\label{l:secondlargest}
Let $0< \delta <1$ and let $p = \frac{1+\delta}{d}$. Then there exists a constant $K_1:=K_1(\delta) >0$ such that the second largest component in $Q^d_p$ has order at most $K_1d$.
\end{lemma}

We will also use the following consequence of Lemma \ref{l:secondlargest}.

\begin{lemma}\label{l:outsidegiant}
Let $0 < \delta_1 <1$ and $\delta_2 \ll \delta_1$, let $q_1 = \frac{1+\delta_1}{d}$ and $q_2= \frac{\delta_2}{d}$ and let $L'_1$ and $L_1$ be the largest components in $Q_1 := Q^d_{q_1}$ and $Q_2:= Q_1 \cup Q^d_{q_2}$, respectively. Given a vertex $v \in V(L'_1)$, let $C_v$ be the set of vertices which are contained in some component of $L_1 - L'_1$ which is adjacent to $v$ in $Q_2$. Then there exists a constant $K_2:=K_2(\delta_1) >0$ such that whp $|C_v| \leq K_2 d$ for every $v \in V(L'_1)$.
\end{lemma}

\begin{proof}
We first note that by, Lemma \ref{l:secondlargest}, there exists a constant $K_1:=K_1(\delta_1)$ such that whp every component of $Q_1$ except $L'_1$ has order at most $K_1d$. Let $K_2 \gg \delta^{-1}_1$.

Suppose that there is some vertex $v \in V(L'_1)$ such that $|C_v|\geq K_2 d$. We note that $C_v \cup \{v\}$ is connected in $Q_2$, and $C_v$ is the disjoint union of some set $\{C_1,\ldots, C_r\}$ where each $C_i$ is the vertex set of some component of $Q_1$, each of which has size at most $K_1 d$. It follows that there must be some subset $\hat{C} \subseteq C_v$ such that $\hat{C} \cup \{v\}$ is connected in $Q_2$, $K_2 d \leq |\hat{C}| \leq (K_1 + K_2)d$ and $\hat{C}$ is the union of some subset of $\{C_1,C_2, \ldots, C_r\}$.

In particular, there is some spanning tree $T$ of $\hat{C} \cup \{v\}$, all of whose edges are present in $Q_2$, such that no edge in the edge-boundary of $V(T) \setminus \{v\}$ is present in $Q_1$. 

Let us bound the probability that such a tree of size $k$ exists in $Q_2$ for each \[
K_2d +1 \leq k \leq (K_1 + K_2)d+1.\] We can fix such a tree $T$ by choosing a root vertex $v$ and choosing one of the at most $(ed)^{k-1}$ possible rooted trees of size $k$ with root $v$ in $Q^d$, where we have bounded the number of possible trees by Lemma \ref{l:treecount}.

Now, $T$ has $k-1$ edges and by Theorem \ref{t:iso}  there are at least $(k-1)(d-\log_2 (k-1))$ edges in the edge-boundary of $V(T) \setminus \{v\}$. Note that each edge is in $Q_2$ with probability at most $(q_1 + q_2)$ and each edge is not in $Q_1$ with probability $(1-q_1)$, and that that, whilst these two events are not necessarily independent, they are clearly negatively correlated. 

It follows by the union bound that the probability that such a tree of size $k$ exists in $Q_2$ is at most

\begin{equation*}
n(ed)^{k-1} \left(q_1 + q_2\right)^{k-1} \left(1- q_1\right)^{(k-1)(d -\log_2 (k-1))}.
\end{equation*}

In particular, the probability that such a tree exists for $k \in I:=[ K_2 d +1,(K_1 + K_2)d+1]$ is at most
\begin{align*}
&n\sum_{k\in I} (e(1+\delta_1 + \delta_2))^{k-1}\left(1- \frac{1+\delta_1}{d}\right)^{(1-o(1))(k-1)d}\\
&\leq  n\sum_{k\in I} \exp\left( (k-1)\left( 1 + \log(1+\delta_1 + \delta_2) - (1-o(1))(1+\delta_1) \right)\right)\\
&\leq n \sum_{k\in I} \exp\left( -\frac{(k-1)\delta^2_1}{5}\right) = o(1),
\end{align*}
where we used that $\log (1+ \delta_1 + \delta_2) \leq \delta_1 - \frac{\delta^2_1}{4}$ for all $\delta_1 \in (0,1)$, since $\delta_2 \ll \delta_1$ and that ${K_2d\delta_1^2  \gg d}$, since $K_2 \gg \delta^{-1}_1$.
\end{proof}

The next lemma says that whp the giant component of $Q^d_p$ is in some sense `dense' in the hypercube $Q^d$.

\begin{lemma}\label{l:dense}
Let $\delta > 0$ and let $p = \frac{1+\delta}{d}$. Then there exists a constant $c > 0$ such that whp every vertex in $Q^d$ is at distance at most two from at least $c d^2$ vertices in the largest component $L_1$ of $Q^d_p$.
\end{lemma}
\begin{proof}
Let us choose some constant $c' \ll \delta$. Fix an arbitrary vertex $v \in V\left(Q^d\right)$, which without loss of generality we may assume to be the origin. Let $k = c' d$ and let $d' = d-k = (1-c')d$. We define $s=\binom{k}{2}$ pairwise disjoint subcubes of dimension $d'$ at distance at most two from $v$ given by fixing some pair of $1$s in the first $k$ coordinates and varying the last $d'$ coordinates. Let these cubes be $Q(1), \ldots, Q(s)$ and let $v_i$ be the vertex in $Q(i)$ at distance two from $v$.

Now, since $c' \ll \delta$, $p$ is still supercritical in $Q^{d'}$ and so, by Theorem \ref{t:AKSfine} and the fact that $Q^d$ is transitive, there is some constant $\alpha >0$ such that probability that $v_i$ is contained in the largest component of $Q(i)_p$ is at least $\alpha$, and these events are independent for different $i$. Hence, by Lemma \ref{l:Chernoff}, with probability at least $1-\exp\left( -c' s \right)$ at least $\frac{\alpha s}{2}$ of the $v_i$ are contained in the largest component of $Q(i)_p$.

Furthermore, again by Theorem \ref{t:AKSfine} and Lemma \ref{l:Chernoff}, whp each $Q(i)_p$ contains a component whose order is $\Omega\left(2^{d'}\right)$ and again these events are independent for different $i$, and hence with probability at least $1-\exp\left(-c' s \right)$ at least $(1- \frac{\alpha}{4})s$ of the $Q(i)_p$ contain a component whose order is $\Omega\left(2^{d'}\right)$. 

It follows that with probability at least $1-2\exp(- c' \alpha)$ at least $\frac{\alpha s}{4} := c d^2$ of the $v_i$ are contained in a component in $Q(i)_p$ whose order is $\Omega\left(2^{d'}\right)$. Hence $v$ is within distance two of at least $c d^2$ vertices lying in components in $Q^d_{p}$ whose order is $\Omega\left(2^{d'}\right)$ with probability at least ${1-2\exp(-c' s)} = 1-o\left(n^{-1}\right)$.

Then, by the union bound, whp every vertex in $Q^d$ is within distance two of at least $c d^2$ many vertices lying in components in $Q^d_{p}$ whose order is $\Omega\left(2^{d'}\right)$. However, by Lemma \ref{l:secondlargest}, whp there is a unique component $L_1$ in $Q^d_{p}$ whose order is superlinear in $d$, and so whp every vertex in $Q^d$ is within distance two of at least $c d^2$  vertices in $L_1$.
\end{proof}

With these results in hand, let us briefly sketch the strategy to prove Theorem \ref{t:expansionsimple1} about the vertex-expansion properties of the giant component of $Q^d_p$, where $p=\frac{1+\epsilon}{d}$. We will use a sprinkling argument, viewing $Q^d_p$ as the union of two independent random subgraphs $Q^d_{q_1}$ and $Q^d_{q_2}$ where we have chosen $q_1 = \frac{1+\delta_1}{d}$ and $q_2 = \frac{\delta_2}{d}$ such that $(1-q_1)(1-q_2)=1-p$ and $\delta_2 \ll \delta_1$. Note that, in particular, $\delta_1$ is approximately $\epsilon$, and so $q_1$ still lies in the supercritical regime. Let us denote by $L_1'$ and $L_1$ the largest components in $Q_1:=Q^d_{q_1}$ and $Q_2 := Q_1 \cup Q^d_{q_2}$, respectively, noting that $Q_2 \sim Q^d_p$.

Given a partition of $V(L'_1)$ into two disjoint subsets $A,B$ such that $|A|,|B| \geq t$, it is relatively easy to show that whp there is a large family of vertex-disjoint $A\textrm{-}B$-paths of length at most $5$ in $Q^d_{q_2}$. Indeed, by Lemma \ref{l:dense} every vertex in $Q^d$ is within distance two of $L'_1$ and so we can extend $A,B$ to a partition of $V(Q^d)$ into two pieces $A' \supseteq A$ and $B' \supseteq B$ such that every vertex in $A'$ is within distance two of $A$ and every vertex in $B'$ is within distance two of $B$. By Theorem \ref{t:iso} there are many edges in $Q^d$ between $A'$ and $B'$ and each such edge can be extended to an $A\textrm{-}B$-path in $Q^d$ of length at most $5$. Very naively, we can thin this family of paths out to a vertex-disjoint family using the fact that $\Delta(Q^d) = d$ whilst retaining an $\Omega(d^{-6})$ proportion of them.  Furthermore, after sprinkling we expect about an $\Omega(d^{-5})$ proportion of these paths to be contained in $Q^d_{q_2}$. Then, as long as $t$ is large enough, the Chernoff bound will imply that whp there is a large vertex-disjoint family of $A\textrm{-}B$-paths of length at most $5$ in $Q^d_{q_2}$. 

In fact, our actual argument will be a bit more precise, to enable us to find a larger family of paths. However, the probability of failure in these arguments will not be small enough to deduce from a union bound that this holds for \emph{all} such partitions of $V(L'_1)$.

Instead, we can use Lemma \ref{l:treedecomp}, with $\ell$ being some small power of $d$, to split $L'_1$ into a collection $\mathcal{C}$ of connected pieces each having polynomial size in $d$. If these pieces are large enough, then there will be sufficiently few partitions of $\mathcal{C}$ into two pieces that the probability bound from the argument above will be effective, and we can deduce that whp whenever we partition $\mathcal{C}$ into two parts there will be a large family of vertex-disjoint paths between them in $Q^d_{q_2}$.

Furthermore, by Lemma \ref{l:outsidegiant} we may assume that for any vertex $v \in V(L'_1)$ there are only a small number of vertices contained in the components of $R:=L_1 -L'_1$ which are adjacent to $v$ in $Q_2$.  

Suppose then that $S$ is some subset of $V(L_1)$. We split $S$ into three pieces:
\begin{itemize}
\item $S_1$ is the set of vertices which lie in components of $R$;
\item $S_2$ is the set of vertices which are contained in pieces $C \in \mathcal{C}$ such that $C \cap S \neq \emptyset$ and $C \setminus S \neq \emptyset$;
\item $S_3$ is the set of vertices which are contained in pieces $C$ such that $C \subseteq S$.
\end{itemize}

If $S_1$ is large, then since each vertex in $L'_1$ is only adjacent in $Q_2$ to components in $R$ with a small total volume, we can greedily choose a large disjoint family $\{C(x) \colon x \in X\}$ of components of $R$ which all meet $S$, each of which is adjacent in $Q_2$ to a unique vertex $x \in L'_1$.

For each $x \in X$, either $x \in S$, or there is some vertex in the neighbourhood of $S$ in $C(x) \cup \{x\}$. In particular, either $S$ has a large neighbourhood, or $S \cap V(L'_1)= S_2 \cup S_3$ is large.

Similarly, if $S_2$ is large then, since each piece in $\mathcal{C}$ is small, $S_2$ contains vertices in many pieces of $\mathcal{C}$, and for each such piece $C$ we have that $C \setminus S \neq \emptyset$. However, since each piece $C \in \mathcal{C}$ is connected in $Q_1$, each piece such that $C \cap S_2 \neq \emptyset$ contains some vertex in the neighbourhood of $S$, and so the neighbourhood of $S$ is large.

Hence, we may assume that $S_3$ is large and $S_2$ is small. In this case, we look at the partition of $\mathcal{C}$ given by pitting the pieces contained in $S_3$ against the rest. By the above argument whp there is a large family of vertex-disjoint paths between these two partition classes and, since $S_2$ is small, not many of these meet $S_2$. Every path which does not meet $S_2$ starts in $S_3 \subseteq S$ and ends in $S^c$, and so contains some vertex in the neighbourhood of $S$. Hence, in every case we can conclude that $S$ has a large neighbourhood.

Let us begin by proving the following result, which guarantees the likely existence of a large family of vertex-disjoint paths between the two parts of any fixed non-trivial partition of $V(L'_1)$. Note that, since there are subsets $A$ of $V(Q^d)$ whose edge-boundary in $Q^d$ is as small as $|A|(d-\log_2|A|)$, for example subcubes, we cannot hope to guarantee the likely existence of a family of paths from $A$ to $A^c$ in the random subgraph $Q^d_p$ with $p = O\left(\frac{1}{d}\right)$ of size larger than $O\left(|A|\left(1-\frac{\log_2|A|}{d}\right)\right)$. Hence, the following lemma is optimal up to a multiplicative constant.

\begin{lemma}\label{l:familyofpaths}
Let $\delta,c > 0$, let $q = \frac{\delta}{d}$, let $L \subseteq Q^d$ be such that every vertex in $Q^d$ is at distance at most two from at least $c d^2$ vertices in $L$ and let $A \cup B = V(L)$ be a partition of $V(L)$ with $\min \{|A|,|B|\} = t$. Then there exists a constant $c' >0$ such that there exists a family of ${c' t \left(1 - \frac{\log_2 t}{d}\right)}$ vertex-disjoint $A\textrm{-}B$-paths of length at most five in $Q^d_q$ with probability at least $1 - \exp \left( -c' t \left(1 - \frac{\log_2 t}{d}\right)\right)$.
\end{lemma}
\begin{proof}
Throughout the proof we will introduce a sequence of constants $c_1, c_2 ,c_3, \ldots$ under the assumption that each $c_i$ is sufficiently small in terms of the preceding $c_j$, $\delta$ and $c$.

Let $s:= t (d - \log_2 t)$ and let us define
\[
\hat{N}(A) = \left\{ v \in V\left(Q^d\right) \colon v \not\in A \text{ and } \left|N_{Q^d}(v) \cap A\right| \geq c_1 d \right\},
\]
and similarly
\[
\hat{N}(B) = \left\{ v \in V\left(Q^d\right) \colon v \not\in B \text{ and } \left|N_{Q^d}(v) \cap B\right| \geq c_1 d  \right\}.
\]

We first note that we may assume that there are at most $c_6 s$ edges between $A$ and $B$ in $Q^d$, since otherwise, by Lemma \ref{l:matching}, with probability at least $1 - \exp\left(-c_7 sd^{-1}\right)$ there will be a matching of size at least $c_7 sd^{-1}$ between $A$ and $B$ in $Q^d_q$. In particular, we can assume that 
\begin{equation}\label{e:hatneighbd}
|B \cap \hat{N}(A)| \leq c_5 sd^{-1} \qquad \text{    and    } \qquad|A \cap \hat{N}(B)| \leq c_5 sd^{-1}.
\end{equation}

By assumption every vertex in $Q^d$ is at distance at most two from at least $c d^2$  vertices in $L$. Hence, we can partition the vertices of $Q^d$ into two disjoint subsets $A'$ and $B'$ such that $A \subseteq A'$ and $B \subseteq B'$, each vertex in $A' \setminus A$ is within distance two of at least $\frac{c d^2}{2}$ vertices in $A$ and each vertex in $B' \setminus B$ is within distance two of at least $\frac{c d^2}{2}$ vertices in $B$.

Since $|A'|,|B'| \geq \min \{|A|,|B|\} = t$, it follows from Theorem \ref{t:iso} that there is a set $F$ of $s$ edges between $A'$ and $B'$. Then, at least one of the following four cases happens:

\begin{enumerate}[i)]
\item\label{i:Alarge} At least $\frac{s}{4}$ edges of $F$ have an endpoint in $A$;
\item\label{i:Aneighbourlarge} At least $\frac{s}{4}$ edges of $F$ have an endpoint in $A' \cap \hat{N}(A)$;
\item\label{i:AmeetsN(B)} At least $\frac{s}{4}$ edges of $F$ have an endpoint in $A' \cap \hat{N}(B)$;
\item\label{i:Aactual} At least $\frac{s}{4}$ edges of $F$ have an endpoint in $A_0$, where $A_0 = A' \setminus (A \cup \hat{N}(A)\cup \hat{N}(B))$.
\end{enumerate}

We will see that case \ref{i:Aactual} is the most complicated case, so let us assume for now that cases \ref{i:Alarge}--\ref{i:AmeetsN(B)} do not hold. We will indicate briefly how to deal with the other cases at the end. If we let $F' \subseteq F$ be a set of $\frac{s}{4}$ edges, each of which has an endpoint in $A_0$, then again at least one of the following four cases happens:
\begin{enumerate}[I)]
\item\label{i:Blarge} At least $\frac{s}{16}$ edges of $F'$ have an endpoint in $B$;
\item\label{i:Bneighbourlarge} At least $\frac{s}{16}$ edges of $F'$ have an endpoint in $B' \cap \hat{N}(B)$;
\item\label{i:BmeetsN(A)} At least $\frac{s}{16}$ edges of $F'$ have an endpoint in $B' \cap \hat{N}(A)$;
\item\label{i:Bactual} At least $\frac{s}{16}$ edges of $F'$ have an endpoint in $B_0$, where $B_0 = A' \setminus (A \cup \hat{N}(A)\cup \hat{N}(B))$.
\end{enumerate}
Again we will assume for now that that cases \ref{i:Blarge}--\ref{i:BmeetsN(A)} do not hold. 

We will construct our family of paths using a sequence of matchings: the first $M_1$ between $A'$ and $B'$; then $M_2$ and $M_3$ joining some subset of the endpoints of $M_1$ to $\hat{N}(A)$ and $\hat{N}(B)$, respectively; and then finally $M_4$ and $M_5$ joining some subset of the endpoints of these matchings to $A$ and $B$, respectively. See Figure \ref{f:matchings}.

\begin{figure}[ht!]
\center
\begin{tikzpicture}[line cap=round,line join=round,>=triangle 45,scale=0.7]
\draw [rotate around={90:(4,0)},line width=1pt] (4,0) ellipse (3.16cm and 1cm);
\draw [line width=1pt] (7,0) circle (1cm);
\draw [line width=1pt,gray] (6.6,-0.6)-- (4,-1.5);
\draw [line width=1pt,gray] (6.4,-0.2)-- (4,-0.5);
\draw [line width=2pt] (6.4,0.2)-- (4,0.5);
\draw [line width=1pt,gray] (6.6,0.6)-- (4,1.5);
\draw [line width=1pt,gray] (4,-2.5)-- (1.5,-3);
\draw [line width=1pt,gray] (4,-1.5)-- (1.5,-1.7);
\draw [line width=1pt,gray] (4,-0.5)-- (1.5,-0.6);
\draw [line width=2pt] (4,0.5)-- (1.5,0.6);
\draw [line width=1pt,gray] (4,1.5)-- (1.5,1.7);
\draw [line width=1pt,gray] (4,2.5)-- (1.5,2.8);
\draw [line width=1pt] (-7,0) circle (0.9930680722671233cm);
\draw [rotate around={90:(-4,0)},line width=1pt] (-4,0) ellipse (3.16cm and 1cm);
\draw [line width=2pt] (-6.4,0.2)-- (-4,0.5);
\draw [line width=1pt,gray] (-6.4,-0.2)-- (-4,-0.5);
\draw [line width=1pt,gray] (-6.6,-0.6)-- (-4,-1.5);
\draw [line width=1pt,gray] (-6.6,0.6)-- (-4,1.5);
\draw [line width=1pt,gray] (-4,2.5)-- (-1.5,2.8);
\draw [line width=1pt,gray] (-4,1.5)-- (-1.5,1.7);
\draw [line width=2pt] (-4,0.5)-- (-1.5,0.6);
\draw [line width=1pt,gray] (-4,-0.5)-- (-1.5,-0.6);
\draw [line width=1pt,gray] (-4,-1.5)-- (-1.5,-1.7);
\draw [line width=1pt,gray] (-4,-2.5)-- (-1.5,-3);
\draw [line width=1pt,gray] (-1.5,-3.9)-- (1.5,-3.9);
\draw [line width=1pt,gray] (-1.5,-3)-- (1.5,-3);
\draw [line width=1pt,gray] (-1.5,-1.7)-- (1.5,-1.7);
\draw [line width=1pt,gray] (-1.5,-0.6)-- (1.5,-0.6);
\draw [line width=2pt] (-1.5,0.6)-- (1.5,0.6);
\draw [line width=1pt,gray] (-1.5,1.7)-- (1.5,1.7);
\draw [line width=1pt,gray] (-1.5,2.8)-- (1.5,2.8);
\draw [line width=1pt,gray] (-1.5,3.9)-- (1.5,3.9);
\draw [line width=1pt] (-9,5)-- (-1,5);
\draw [line width=1pt] (-1,5)-- (-1,-5);
\draw [line width=1pt] (-1,-5)-- (-9,-5);
\draw [line width=1pt] (-9,-5)-- (-9,5);
\draw [line width=1pt] (9,5)-- (1,5);
\draw [line width=1pt] (1,5)-- (1,-5);
\draw [line width=1pt] (1,-5)-- (9,-5);
\draw [line width=1pt] (9,-5)-- (9,5);
\draw (-8.7,4.8) node[anchor=north west] {$A'$};
\draw (8,4.8) node[anchor=north west] {$B'$};
\draw (-7.7,0.5) node[anchor=north west] {$A$};
\draw (7,0.5) node[anchor=north west] {$B$};
\draw (-4.8,4.3) node[anchor=north west] {$\hat{N}(A)$};
\draw (3.2,4.3) node[anchor=north west] {$\hat{N}(B)$};
\draw (-5.9,-5.2) node[anchor=north west] {$M_4$};
\draw (-3.1,-5.2) node[anchor=north west] {$M_2$};
\draw (-0.4,-5.2) node[anchor=north west] {$M_1$};
\draw (2.3,-5.2) node[anchor=north west] {$M_3$};
\draw (5.,-5.2) node[anchor=north west] {$M_5$};
\begin{scriptsize}
\draw[fill]  (6.4,0.2) circle (2.5pt);
\draw[fill]  (6.4,-0.2) circle (2.5pt);
\draw[fill]  (6.6,-0.6) circle (2.5pt);
\draw[fill]  (6.6,0.6) circle (2.5pt);
\draw[fill]  (6.4,0.2) circle (2.5pt);
\draw[fill]  (6.4,-0.2) circle (2.5pt);
\draw[fill]  (6.6,-0.6) circle (2.5pt);
\draw[fill]  (6.6,0.6) circle (2.5pt);
\draw[fill]  (4,0.5) circle (2.5pt);
\draw[fill]  (4,-0.5) circle (2.5pt);
\draw[fill]  (4,1.5) circle (2.5pt);
\draw[fill]  (4,-1.5) circle (2.5pt);
\draw[fill]  (4,2.5) circle (2.5pt);
\draw[fill]  (4,-2.5) circle (2.5pt);
\draw[fill]  (1.5,0.6) circle (2.5pt);
\draw[fill]  (1.5,-0.6) circle (2.5pt);
\draw[fill]  (1.5,-1.7) circle (2.5pt);
\draw[fill]  (1.5,1.7) circle (2.5pt);
\draw[fill]  (1.5,2.8) circle (2.5pt);
\draw[fill]  (1.5,3.9) circle (2.5pt);
\draw[fill]  (1.5,-3) circle (2.5pt);
\draw[fill]  (1.5,-3.9) circle (2.5pt);
\draw[fill]  (-1.5,3.9) circle (2.5pt);
\draw[fill]  (-1.5,2.8) circle (2.5pt);
\draw[fill]  (-1.5,1.7) circle (2.5pt);
\draw[fill]  (-1.5,0.6) circle (2.5pt);
\draw[fill]  (-1.5,-0.6) circle (2.5pt);
\draw[fill]  (-1.5,-1.7) circle (2.5pt);
\draw[fill]  (-1.5,-3) circle (2.5pt);
\draw[fill]  (-1.5,-3.9) circle (2.5pt);
\draw[fill]  (-4,2.5) circle (2.5pt);
\draw[fill]  (-4,1.5) circle (2.5pt);
\draw[fill]  (-4,0.5) circle (2.5pt);
\draw[fill]  (-4,-0.5) circle (2.5pt);
\draw[fill]  (-4,-1.5) circle (2.5pt);
\draw[fill]  (-4,-2.5) circle (2.5pt);
\draw[fill]  (-6.4,-0.2) circle (2.5pt);
\draw[fill]  (-6.4,0.2) circle (2.5pt);
\draw[fill]  (-6.6,0.6) circle (2.5pt);
\draw[fill]  (-6.6,-0.6) circle (2.5pt);
\end{scriptsize}
\end{tikzpicture}\caption{The sequence of matchings $M_1$--$M_5$, with one of the paths resulting from the concatenation highlighted in bold. Note that, unlike in the diagram, it may be the case that $\hat{N}(A)$ meets $B'$, or $\hat{N}(B)$ meets $A'$.}\label{f:matchings}
\end{figure}
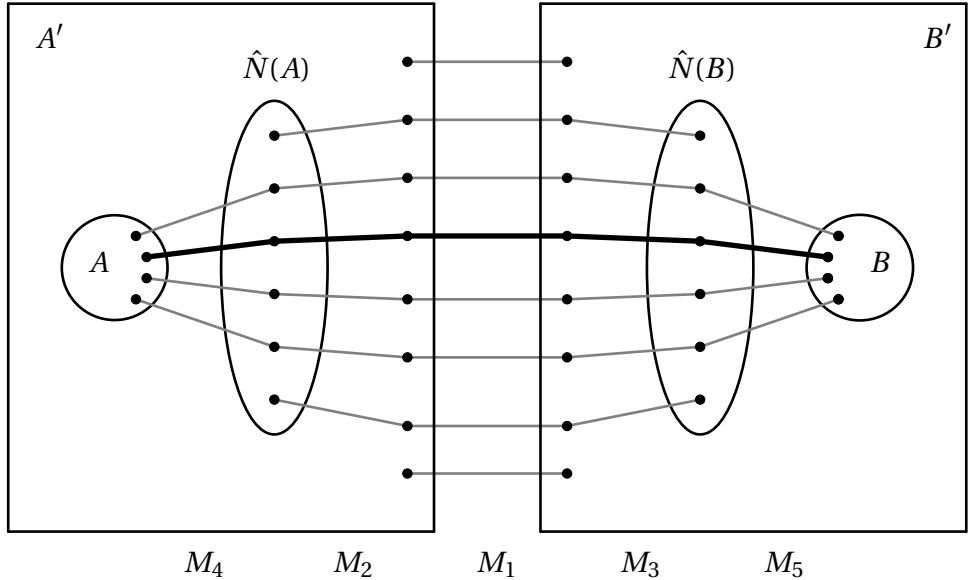

Let $F_1\subseteq F'$ be a set of $\frac{s}{16}$ edges whose endpoints lie in $A_0$ and $B_0$. Then, by Lemma \ref{l:matching}, with probability at least $1- \exp\left( - c_2 s d^{-1}\right)$ there is a matching $M_1$ contained in $(F_1)_q$ of size at least $c_2 s d^{-1}$. Let $A_1 \subseteq A_0$ and $B_1 \subseteq B_0$ be the endpoints of this matching.

Since each vertex in $A_1$ is within distance two in $Q^d$ of at least $\frac{c d^2}{2}$ vertices in $A$, $\Delta(Q^d) = d$, and no vertex in $A_1$ is in $A$ or $\hat{N}(A)$, it follows that we can fix, for each vertex $u_i \in A_1$, a star $T_i$ in $Q^d$ rooted at $u_i$ with $c_1 d$ leaves, such that each leaf is adjacent in $Q^d$ to at least $c_1 d$ vertices in $A$ and no leaf is in $A$, and so each leaf is in $\hat{N}(A)$. Note that, since each edge in these stars is from $A_1 \subseteq A'$ to $\hat{N}(A)$ and $B_0 \subseteq B'$ is disjoint from $A' \cup \hat{N}(A)$, it follows that none of the edges in these stars lie in $F_1$, each edge of which meets $B_0$.

Let $C_1$ be the set of vertices which are leaves in some star $T_i$ and let $F_2$ be the set of edges between $A_1$ and $C_1$ contained in these stars. Then $F_2 \cap F_1 = \emptyset$ and $|F_2| = |A_1|c_1 d = c_1 c_2 s$.

Then, again by Lemma \ref{l:matching}, with probability at least $1 - \exp\left(-c_3 s d^{-1} \right)$ there is a matching $M_2$ contained in $(F_2)_q$ of size at least $c_3 s d^{-1}$. Let $A_2 \subseteq A_1$ be the endpoints of this matching in $A_1$ and let $B_2$ be the set of vertices in $B_1$ joined to a vertex in $A_2$ via the matching $M_1$. We now make a similar argument for the vertices in $B_2$.

Namely, since each vertex in $B_2$ is within distance two in $Q^d$ of at least $\frac{c d^2}{2}$ vertices in $B$, $\Delta(Q^d) = d$, and no vertex in $B_2$ is in $B$ or adjacent to $B$, it follows that we can fix, for each vertex $v_i \in B_2$ a star $T'_i$ in $Q^d$ rooted at $v_i$ with $c_1 d$ leaves, such that each leaf is adjacent in $Q^d$ to at least $c_1 d$ vertices in $B$ and no leaf is in $B$, and so each leaf is in $\hat{N}(B)$. Note that, as before, since each edge in these stars is from $B_2 \subseteq B'$ to $\hat{N}(B)$ and $A_0 \subseteq A'$ is disjoint from $B' \cup \hat{N}(B)$, none of the edges in these stars lie in $F_1$ or $F_2$, each edge of which has an endpoint in $A_0$.

Hence, if we let $D_1$ be the set of vertices which are leaves in some star $T'_i$ and consider the set of edges $F_3$ between $B_2$ and $D_1$ contained in these stars, then $F_3 \cap (F_1 \cup F_2) = \emptyset$ and $|F_3| = |B_2|c_1 d = c_1 c_3 s$. Again, by Lemma \ref{l:matching}, we can conclude that with probability at least $1 - \exp\left(-c_4 s d^{-1} \right)$ there is a matching $M_3$ contained in $(F_3)_q$ of size at least $2c_4 s d^{-1}$.

By combining the matchings $M_1,M_2$ and $M_3$, we obtain a family of vertex-disjoint paths $\mathcal{P}' = \{P_1, \ldots, P_r\}$ of size $2c_4 s d^{-1}$, where each $P_i$ has one endpoint $x_i \in C_1$, which by construction of $C_1$ is adjacent in $Q^d$ to $c_1 d$  vertices in $A$, and a second endpoint $y_i \in D_1$, which similarly is adjacent in $Q^d$ to $c_1 d$ vertices in $B$. Let $C_2\subseteq C_1$ and $D_2\subseteq D_1$ be the sets of endpoints of $\mathcal{P}'$.

We note that the set of edges between $C_2$ and $A$ and the set of edges between $D_2$ and $B$ do not intersect with the set of edges $F_1 \cup F_2 \cup F_3$ which we already exposed. Indeed, every edge in $F_1 \cup F_2$ has an endpoint in $A_0$, which by construction is disjoint from $C_2 \subseteq \hat{N}(A)$, $D_2 \subseteq \hat{N}(B)$, $A$ and $B$. Similarly every edge in $F_1 \cup F_3$ has an endpoint in $B_0$, which is again disjoint from $C_2,D_2,A$ and $B$. However, the set of edges between $C_2$ and $A$ might intersect with the set of edges between $D_2$ and $B$.

To deal with this, let $I_1$ be the set of $i$ such that some edge from $x_i \in C_2$ to $A$ coincides with an edge from $D_2$ to $B$. Then, since $A \cap B = \emptyset$, it follows that $x_i \in B$. However, since each $x_i \in C_2 \subseteq \hat{N}(A)$, it follows that $\{x_i \colon i\in I_1\} \subseteq \hat{N}(A) \cap B$ and hence $|I_1| \leq  |\hat{N}(A) \cap B|\leq c_5 sd^{-1}$ by \eqref{e:hatneighbd}.

Similarly, if we let $I_2$ be the set of $i$ such that some edge from $y_i \in D_2$ to $B$ coincides with an edge from $C_2$ to $A$, then we can conclude that $|I_2| \leq c_5 sd^{-1}$.

Hence, if we let $\mathcal{P} = \{ P_i \colon i \in [r] \setminus (I_1 \cup I_2) \}$ and let $C_3$ and $D_3$ be the endpoints of these paths, then $|\mathcal{P}| \geq c_4 s d^{-1}$. Then, there is a set $F_4$ of at least $|\mathcal{P}|c_1 d = c_8 s$ edges between $C_3$ and $A$ and by construction $F_4 \cap (F_1 \cup F_2 \cup F_3) = \emptyset$. Hence, by Lemma \ref{l:matching}, with probability at least $1-\exp\left(-c_9 sd^{-1} \right)$ there is a matching $M_4$ of size $c_9 s d^{-1}$ in $(F_4)_q$ between $C_3$ and $A$. Let $D_4 \subseteq D_3$ be the endpoints of the paths in $\mathcal{P}$ whose other endpoint is an endpoint of an edge in $M_4$.

As before, there is a set $F_5$ of at least $c_{10} s$ edges between $D_4$ and $B$, and by construction ${F_5 \cap (F_1 \cup F_2 \cup F_3 \cup F_4) = \emptyset}$. Therefore, again by Lemma \ref{l:matching}, with probability at least ${1-\exp\left(-c_{11} sd^{-1} \right)}$ there is a matching $M_5$ of size $c_{11} s d^{-1}$ in $(F_5)_q$  between $D_4$ and $B$.

In particular, by combining the matchings $M_1,M_2,M_3,M_4$ and $M_5$ we can construct a family of $c_{11} s d^{-1}$ vertex-disjoint $A\textrm{-}B$-paths in $Q^d_q$.

Note that, throughout the argument we assumed a finite number of whp events occurred, and in each case the probability of failure was at most $\exp\left(-c_{11} s d^{-1}\right)$, and so the conclusion holds with probability at least $1- \exp\left(-c_{12} s d^{-1}\right)$. In particular it follows that the claim holds with $c' = c_{12}$.

If one of the cases \ref{i:Alarge}--\ref{i:AmeetsN(B)} or \ref{i:Blarge}--\ref{i:BmeetsN(A)} holds, then we can avoid building some of the matchings $M_i$. For example, if case \ref{i:Aneighbourlarge} holds, then instead of building $M_2$ and $M_4$ we can instead build a large matching directly from $A_0$ to $A$ before building $M_1$, $M_3$ and $M_5$. Similarly, if case \ref{i:Alarge} holds, then we only need to build $M_1$, $M_3$ and $M_5$. If neither case \ref{i:Alarge} nor \ref{i:Aneighbourlarge} hold, but case \ref{i:AmeetsN(B)} holds, then $A'$ must contain at least $\frac{s}{4d}$ vertices of $\hat{N}(B)$, and so there is a set of at least $\frac{c_4 s}{4}$ edges from $A'$ to $B$, using which we can build a matching of size $c_5s$ using Lemma \ref{l:matching}. We can then extend this matching to a family of $A\textrm{-}B$-paths using matchings $M_2$ and $M_4$ as before. Similar arguments work when one of the cases \ref{i:Blarge}--\ref{i:BmeetsN(A)} holds.

In fact, some possibilities are already excluded by our earlier assumptions, for example we can assume \ref{i:Alarge} and \ref{i:Blarge} do not simultaneously hold since we are assuming there are at most $c_6 s$ edges between $A$ and $B$.
\end{proof}

At this point we have all the necessary tools to prove the following theorem, which clearly implies Theorem \ref{t:expansionsimple1}.

\begin{theorem}\label{t:expansion}
Let $0 < \alpha \ll \epsilon$, let $p=\frac{1+\epsilon}{d}$ and let $L_1$ be the largest component of $Q^d_p$. Then there exists a constant $\beta>0$ such that whp:
\begin{enumerate}[(a)]
\item\label{i:small} every subset $S \subseteq V(L_1)$ satisfying $|S| \leq \frac{|V(L_1)|}{2}$ is such that \[
\left|N_{Q^d_p}(S)\right|  \geq \beta |S|\left(1 - \frac{\log_2 |S|}{d}\right)^2 d^{-3} \geq \beta d^{-5};
\]
\item\label{i:large} every subset $S \subseteq V(L_1)$ satisfying $ \alpha n \leq |S| \leq \frac{|V(L_1)|}{2}$ is such that
\[
\left|N_{Q^d_p}(S)\right|  \geq \beta |S|d^{-2} (\log d)^{-1}.
\]
\end{enumerate}
\end{theorem}
\begin{proof}
Throughout the proof we will introduce a sequence of constants $c_1, c_2 ,c_3, \ldots$ under the assumption that each $c_i$ is sufficiently small in terms of the preceding $c_j$, $\epsilon$ and $\alpha$.

We will argue using a sprinkling argument. Let $c_2 \ll c_1 \ll \alpha$, let $q_2 = \frac{c_2}{d}$ and let $q_1 = \frac{p - q_2}{1-q_2}$. 
Note that, since $c_2 \ll \epsilon$ it follows that $q_1$ is still supercritical. Furthermore, if we let $\gamma$ and $\gamma_1$ be the survival probabilities of the ${\text{Po}(1+\epsilon)}$ and ${\text{Po}(dq_1)}$ branching processes, respectively, then since ${c_2 \ll c_1}$ we may assume that
\begin{equation}\label{e:survival}
\gamma - \gamma_1 \leq c_1.
\end{equation}
We will generate independently two random subgraphs $Q^d_{q_1}$ and $Q^d_{q_2}$ and let $Q_1 := Q^d_{q_1}$ and $Q_2 := Q_1 \cup Q^d_{q_2}$, so that $Q_2 \sim Q^d_p$.

Let us first note a few likely properties of the graphs $Q_1$ and $Q_2$. Firstly, it follows from Lemma \ref{l:degrees} that
\begin{equation}\label{e:degrees}
\text{whp $Q_1$ contains at most $nd^{-4}$ vertices of degree at least $\log d$.}
\end{equation}
Furthermore, if we let $L'_1$ and $L_1$ be the largest components in $Q_1$ and $Q_2$, respectively, then by Theorem \ref{t:AKSfine}
\begin{equation}\label{e:giantsize}
\text{whp $|V(L_1)| = (\gamma \pm c_3) n$ and $|V(L'_1)| = (\gamma_1 \pm c_3) n $.}
\end{equation} 
Note that \eqref{e:survival} and \eqref{e:giantsize} imply that $|V(L'_1)| \geq (\gamma - c_1 - c_3) n \geq \frac{3}{4} |V(L_1)|$.

Given a vertex $v \in V(L'_1)$, let $C_v$ be the set of vertices which are contained in some component of $L_1 - L'_1$ which is adjacent to $v$ in $Q_2$. Then, by Lemma \ref{l:outsidegiant} there exists a constant $K_2:=K_2(dq_1 - 1) >0$ such that
\begin{equation}\label{e:outsidegiant}
\text{whp $|C_v| \leq K_2 d$ for every vertex in $v \in V(L'_1)$,}
\end{equation} 
where we may assume that $K_2^{-1} := c_4 \ll c_3$. Note, in particular, that \eqref{e:outsidegiant} implies that every component in $L_1 - L'_1$ has size at most $K_2 d$.

Moreover, by Lemma \ref{l:dense},
\begin{equation}\label{e:dense}
\text{whp every vertex in $Q^d$ is at distance at most two from at least $c_5 d^2$ vertices in $L'_1$.}
\end{equation}
In what follows, we will assume that \eqref{e:survival}--\eqref{e:dense} hold.

We want to split $L'_1$ into a disjoint family of relatively small, connected pieces, however in order to treat the cases \ref{i:small} and \ref{i:large} we will need to use slightly different families. Given $s \in \mathbb{N}$, let us write 
\[
b(s):= 1 - \frac{\log_2 s}{d}.
\]
\begin{itemize}
\item For case \ref{i:small}, for each $1 \leq s \leq \frac{|V(L_1)|}{2}$ we use Lemma \ref{l:treedecomp} to split $L'_1$ into a family $\mathcal{C}(s)$ of vertex-disjoint connected subgraphs, which we will refer to as \emph{pieces}, such that all pieces have size between  $c_8^{-1}b(s)^{-1} d$ and $c_8^{-1}b(s)^{-1} d^{2}$.
\item Similarly, for case \ref{i:large} we use \eqref{e:degrees} and Lemma \ref{l:treedecomp} to split $L'_1$ into a family $\mathcal{C}'$ of connected pieces such that at most $nd^{-4}$ of the pieces in $\mathcal{C}'$ have size between $c_8^{-1} d$ and $c_8^{-1} d^{2}$ and the rest have size between $c_8^{-1} d$ and $c_8^{-1} d \log d$.
\end{itemize}

We now state two very similar claims about the existence of certain path families, whose proofs we defer to the end of this proof. The first will be useful for case \ref{i:small}.

\begin{claim}\label{c:pathsmall}
Whp for any $1 \leq t \leq 2s \leq |V(L_1)|$ and any partition of $\mathcal{C}(s)$ into two sets $\{\mathcal{C}_A, \mathcal{C}_B\}$, where $A := \bigcup \mathcal{C}_A$ and  $B := \bigcup \mathcal{C}_B$, with $ \min \{|A|,|B| \} =t$, there is a family of at least $c_7 tb(t)$ vertex-disjoint $A\textrm{-}B$-paths of length at most five in $Q^d_{q_2}$.
\end{claim}

The second, which follows by a similar argument, will be useful for case $\ref{i:large}$.

\begin{claim}\label{c:pathlarge}
Whp for any partition of $\mathcal{C}'$ into two sets $\{\mathcal{C}'_A, \mathcal{C}'_B\}$, where $A := \bigcup \mathcal{C}'_A$ and  ${B := \bigcup \mathcal{C}'_B}$, with $\min \{|A|,|B| \} = t \geq \frac{\alpha n}{4} $ there is a family of at least $c_7 td^{-1}$ vertex-disjoint $A\textrm{-}B$-paths of length at most five in $Q^d_{q_2}$.
\end{claim}

Let us further assume that $Q^d_{q_2}$ satisfies the conclusions of Claims \ref{c:pathsmall} and \ref{c:pathlarge}. We will subsequently be able to deduce the claimed expansion properties deterministically.

Let $S \subseteq V(L_1)$ be an arbitrary subset of size $s \leq \frac{|V(L_1)|}{2}$ and let $S_1 := S \cap V(R)$ be the vertices of $S$ which lie in $R:=L_1 - L'_1$.

In order to deal with case \ref{i:small} let us further split $S \cap V(L'_1)$ into two parts as follows:
\begin{itemize}
\item $S_2$ is the set of vertices which are contained in pieces $C \in \mathcal{C}(s)$ such that $S \cap C \neq \emptyset$ and $S  \setminus C \neq \emptyset$;
\item $S_3$ is the set of vertices which are contained in pieces $C \in \mathcal{C}(s)$ such that $C \subseteq S$.
\end{itemize}
Similarly, to deal with case \ref{i:large} we further split $S\cap V(L'_1)$ into two parts as follows:
\begin{itemize}
\item $S'_2$ is the set of vertices which are contained in pieces $C \in \mathcal{C}'$ such that $S \cap C \neq \emptyset$ and $S  \setminus C \neq \emptyset$;
\item $S'_3$ is the set of vertices which are contained in pieces $C \in \mathcal{C}'$ such that $C \subseteq S$.
\end{itemize}

{\bf Case \ref{i:small} :}

Suppose first that $|S_1| \geq \frac{s}{2}$. Then, by \eqref{e:outsidegiant} we can choose some subset $X \subseteq V(L'_1)$ of size at least $\frac{c_4 s}{2d}$
together with a disjoint family $\{C(x) \colon x \in X\}$ of components of $R$ such that $x$ is adjacent to $C(x)$ in $Q_2$ for each $x \in X$ and each $C(x)$ meets $S$.

For each $x \in X$ either $x \in S$, or there is some vertex in $C(x) \cup \{x\}$ which lies in the neighbourhood of $S$. In particular, it follows that either $\left|N_{Q_2}(S)\right|  \geq \frac{c_4 s}{4d}$ or 
\[
|S_2 \cup S_3| = |S \cap V(L'_1)| \geq \min\left\{\frac{s}{2}, \frac{c_4 s}{4d} \right\} = \frac{c_4 s}{4d}.
\]

Suppose then that $|S_2| \geq \frac{c_4c_7s b(s)}{16d}$. Since each piece in $\mathcal{C}(s)$ has size at most $c_8^{-1}b(s)^{-1} d^2$ it follows that $S_2$ contains vertices in at least $\frac{c_4c_7 c_8s b(s)^2}{16 d^3}$ many pieces of $\mathcal{C}$. Since each such piece contains at least one vertex in the neighbourhood of $S$, it follows that $\left|N_{Q_2}(S)\right|  \geq \frac{c_4c_7 c_8 s b(s)^2}{16 d^3}$.

On the other hand, if $|S_2 \cup S_3| \geq \frac{c_4 s}{4d}$ and $|S_2| \leq \frac{c_4c_7s b(s)}{16d}$ then $|S_3| \geq \frac{c_4 s}{8d}$. Let 
\[
\cc{C}_A = \{ C \in \mathcal{C} \colon C \subseteq S\} \qquad \text{ and } \qquad  \cc{C}_B = \mathcal{C} \setminus \cc{C}_A,
\]
and let $A := \bigcup \mathcal{C}_A$ and  $B := \bigcup \mathcal{C}_B$. Then $|A| = |S_3| \geq \frac{c_4 s}{8d}$, and by \eqref{e:survival} and \eqref{e:giantsize}
\[
|B| \geq |V(L'_1)| - |S| \geq |V(L'_1)| - \frac{|V(L_1)|}{2} \geq  \frac{|V(L_1)|}{4} \geq  \frac{c_4 s}{8d}.
\] 
Hence, by Claim \ref{c:pathsmall} there is family of $\frac{c_4c_7 s b\left(\frac{c_4 s}{8d}\right)}{8d} \geq\frac{c_4c_7 s b(s)}{8d} $ vertex-disjoint $A\textrm{-}B$-paths in $Q^d_{q_2}$. Since at most $|S_2| \leq \frac{c_4c_7 s b(s)}{16d}$ of these paths can meet $S_2$, the rest of these paths go from $A \subseteq S$ to $B \setminus S_2 \subseteq S^c$, and so each path contributes a vertex to the neighbourhood of $S$. It follows that $\left|N_{Q_2}(S)\right|  \geq \frac{c_4c_7 s b(s) }{16 d}$. In particular, in all cases 
\[\left|N_{Q_2}(S)\right|=\left|N_{Q_2}(S)\right|  \geq c_{9} s b(s)^2  d^{-3}.\] 

{\bf Case \ref{i:large} :}

We first note that, by \eqref{e:giantsize} and \eqref{e:survival}
\[
|S_1| \leq |V(R)| \leq |V(L_1)| - |V(L'_1)| \leq (\gamma + c_3 - \gamma_1 +c_3)n \leq (c_1 + 2c_3)n
\]
and hence $|S'_2 \cup S'_3| \geq |S| - |S_1| \geq \frac{\alpha n}{2}$.

If $|S'_2| \geq \frac{c_7 \alpha n}{8d}$ then, since at most $nd^{-4} c_8^{-1} d^2 =  c_8^{-1} n d^{-2}$ vertices lie in pieces of $\mathcal{C}'$ of size larger than $c_8^{-1} d \log d$, it follows that $S'_2$ contains vertices in at least $\frac{c_8 c_7 \alpha n}{16d^2 \log d}$ many pieces of $\mathcal{C}'$. Since each piece contains at least one vertex in the neighbourhood of $S$ it follows that $\left|N_{Q_2}(S)\right|  \geq \frac{c_8 c_7 \alpha n}{16d^2 \log d}$. 

On the other hand, if $|S'_2| \leq \frac{c_7 \alpha n}{8d}$ and $|S'_2 \cup S'_3| \geq \frac{\alpha n}{2}$, then $|S'_3| \geq \frac{\alpha n}{4}$. Let 
\[
\cc{C}_A = \{ C \in \mathcal{C}' \colon C \subseteq S\} \qquad \text{ and } \qquad  \cc{C}_B = \mathcal{C}' \setminus \cc{C}_A,
\]
and let $A := \bigcup \mathcal{C}'_A$ and  $B := \bigcup \mathcal{C}'_B$. Then $|A| = |S'_3| \geq \frac{\alpha n}{4}$ and by \eqref{e:giantsize}
\[
|B| \geq |V(L'_1)| - |S| \geq |V(L'_1)| - \frac{|V(L_1)|}{2} \geq  \frac{|V(L_1)|}{4} \geq  \frac{\alpha n}{4}.
\] 
Hence, by Claim \ref{c:pathlarge} there is a family of $\frac{c_7\alpha n}{4d}$ many vertex-disjoint $A\textrm{-}B$-paths in $Q^d_{q_2}$. Since at most $|S'_2| \leq \frac{c_7\alpha n}{8d}$ of these paths can meet $S'_2$, the rest of these paths go from $A \subseteq S$ to $B \setminus S'_2 \subseteq S^c$, and so each path contributes a vertex to the neighbourhood of $S$. It follows that ${\left|N_{Q_2}(S)\right|  \geq \frac{c_7\alpha n}{8d}}$. In particular, in both cases 
\[
\left|N_{Q_2}(S)\right| = \left|N_{Q_2}(S)\right| \geq c_9 n d^{-2} (\log d)^{-1}.
\]

The conclusion of the theorem then follows with $\beta = c_9$.

\begin{proof}[Proof of Claim \ref{c:pathsmall}]
Let us fix an $s$ and such a partition $\{\mathcal{C}_A,\mathcal{C}_B\}$, where $\min \{|A|,|B| \} =t \leq 2s$. Note that, since each piece in $\mathcal{C}(s)$ has size at least $c_8^{-1} b(s)^{-1}d$, it follows that 
\[
\text{$t \geq c_8^{-1} b(s)^{-1}d =:t_{\textrm{min}}$ and $s \geq \frac{t_{\textrm{min}}}{2} =: s_{\textrm{min}}$.}
\] 
Let us suppose that $\mathcal{C}_A$ contains $k$ pieces of $\mathcal{C}$, where 
\[
k_1 := c_8t b(s)d^{-2} \leq k \leq c_8 tb(s)  d^{-1} :=k_2.
\]
Note that there are at most
\[
\binom{|\mathcal{C}|}{k} \leq n^k
\]
partitions of this form. Since $A \cup B = V(L'_1)$, by Lemma \ref{l:familyofpaths} the probability that such a partition does not satisfy the conclusion of the claim is at most $\exp\left(-c_7 tb(t) \right)$, and so by the union bound the probability that some partition does not satisfy the conclusion of the lemma is at most
\begin{align*}
\sum_{s = s_{\textrm{min}}}^{\frac{|V(L_1)|}{2}} \sum_{t = t_{\textrm{min}}}^{2s} \sum_{k=k_1}^{k_2} \exp(-c_7 t b(t)) n^k &\leq \sum_{s = s_{\textrm{min}}}^{\frac{|V(L_1)|}{2}} \sum_{t = t_{\textrm{min}}}^{2s} (k_2-k_1)\exp(-c_7 tb(t)) n^{k_2}\\
&\leq \sum_{s = s_{\textrm{min}}}^{\frac{|V(L_1)|}{2}} \sum_{t = t_{\textrm{min}}}^{2s} c_8 tb(s) d^{-1}\exp(-c_7 tb(t)  ) n^{c_8tb(s)d^{-1}}\\
&\leq \sum_{s = s_{\textrm{min}}}^{\frac{|V(L_1)|}{2}}\sum_{t = t_{\textrm{min}}}^{2s} \exp\left(-\frac{c_7}{2} tb(t)\right) \\
&\leq \sum_{s = s_{\textrm{min}}}^{\frac{|V(L_1)|}{2}} 2 \exp\left(-\frac{c_7}{2} t_{\textrm{min}} b(t_{\textrm{min}})\right) = o(1),
\end{align*}
where in the above we used that, since $b(s) \leq 2 b(2s) \leq 2b(t)$,
\[
c_8tb(s)d^{-1}n^{c_8tb(s)d^{-1}} \leq \exp\left(c_8tb(s)\right) \leq \exp\left(\frac{c_7 tb(t) }{2} \right),
\]
and that $\frac{c_7}{2}t_{\textrm{min}} b(t_{\textrm{min}}) \geq  \frac{c_7 t_{\textrm{min}}}{8} \geq d$.
\end{proof}

\begin{proof}[Proof of Claim \ref{c:pathlarge}]
Let us fix such a partition $\left\{\mathcal{C}'_A,\mathcal{C}'_B\right\}$ where $\min \{|A|,|B| \} =t \geq \frac{\alpha n}{4}$. Since there are at most $c_8nd^{-1}$ pieces in $\mathcal{C}'$, there are at most
\[
2^{c_8nd^{-1}}
\]
many partitions of $\mathcal{C}'$. Since $A \cup B = V(L'_1)$ and $1- \frac{\log_2 t}{d} \geq d^{-1}$, by Lemma \ref{l:familyofpaths} the probability that such a partition does not satisfy the conclusion of the claim is at most $\exp\left(-c_7 t d^{-1}\right)$, and so by the union bound the probability that any partition does not satisfy the conclusion of the lemma is at most
\begin{align*}
\sum_{t \geq \frac{\alpha n}{4}}  \exp\left(-c_7 td^{-1}\right) 2^{c_8nd^{-1}} &\leq \sum_{t \geq \frac{\alpha n}{4}}  \exp\left(-\frac{c_7 td^{-1}}{2}\right) = o(1).
\end{align*}
\end{proof}
\end{proof}
\begin{remark}\label{r:biggerthan1/2}
The restriction to subsets of size at most $\frac{|V(L_1)|}{2}$ in Theorem \ref{t:expansion} is mostly for ease of presentation. It is relatively simple to see that subsets of $V(L_1)$ of size up to $(1-2\alpha)|V(L_1)|$ will have similar expansion properties.

Indeed, let $S$ be any subset of $V(L_1)$ such that $\frac{|V(L_1)|}{2} \leq |S| \leq (1-2\alpha)|V(L_1)|$ and whose boundary satisfies ${|N(S)| \leq \beta|S|d^{-2}(\log d)^{-1} \leq \alpha n}$. It follows that $W:= V(L_1) \setminus (S \cup N(S))$ is such that $\alpha n \leq |W| \leq \frac{|V(L_1)|}{2}$ and $N(W) \subseteq N(S)$. 

In particular, under the conclusion of Theorem \ref{t:expansion} \ref{i:large},
\begin{equation}\label{e:bdrylargeset}
|N(S)| \geq |N(W)| \geq \beta |W|d^{-2}(\log d)^{-1} \geq \frac{\alpha \beta}{1- 2\alpha} |S| d^{-2}(\log d)^{-1} \geq \beta |S| d^{-5}.
\end{equation}
\end{remark}

\begin{proof}[Proof of Theorem \ref{t:expansionsimple1}]
The claim follows immediately from Theorem \ref{t:expansion} \ref{i:small}.
\end{proof}

We note that it is relatively easy to see with a similar argument that Claims \ref{c:pathsmall}--\ref{c:pathlarge} imply that whp the subset $V(L'_1)\subseteq V(L_1)$ has good expansion properties, in fact slightly better than the expansion we have for the whole giant component $L_1$, although this expansion may happen `outside' of $V(L'_1)$.

\begin{lemma}\label{l:L'_1}
Let $\delta,\epsilon > 0$, let $q_1 = \frac{1+\epsilon}{d}$ and let $q_2 = \frac{\delta}{d}$. Let $L'_1$ be the largest component of $Q^d_{q_1}$ and let $Q_2 := Q^d_{q_1} \cup Q^d_{q_2}$. Then whp every subset $S \subseteq V(L'_1)$ of size $|S| \leq \frac{|V(L'_1)|}{2}$ is such that 
\[
\left|N^5_{Q_2}(S) \cap V(L'_1)\right| \geq \beta |S| \left( 1 - \frac{\log_2 |S|}{d}\right)^2 d^{-2}.
\]
\end{lemma}
\begin{remark}
As with Theorem \ref{t:expansion}, there will be somewhat stronger expansion if we additionally assume that $|S| \geq \alpha n$.
\end{remark}
\begin{proof}[Proof of Lemma \ref{l:L'_1}]
We sketch the argument below, without keeping careful track of the constants. 

Given a subset $S \subseteq V(L'_1)$, we can consider the partition $S = S_2 \cup S_3$ as in the proof of case \ref{i:small} of Theorem \ref{t:expansion}. Each piece $C \in \mathcal{C}(s)$ containing a vertex of $S_2$ also contains a vertex in $N_{Q_2}(S) \cap V(L'_1)$, and so, since each piece in $\mathcal{C}(s)$ has size $O(b(s)^{-1}d^2)$, we are done if ${|S_2| = \Omega\left(sb(s) \right)}$. Hence, we may assume that $|S_2| = o\left(sb(s) \right)$ and so $|S_3| = \Omega(s)$. 

If we let $\mathcal{C}_A$ be the set of pieces of $\mathcal{C}(s)$ contained in $S$ and $\mathcal{C}_B$ be the rest, where $A:= \bigcup \mathcal{C}_A$ and $B:= \bigcup \mathcal{C}_B$, then, as before,  $|A| = |S_3| = \Omega(s)$ and $|B| = \Omega(s)$ and so by Claim \ref{c:pathsmall} whp there is a family of $\Omega\left(sb(s)\right)$ vertex-disjoint $A\textrm{-}B$-paths of length at most five in $Q^d_{q_2}$. Since at most $|S_2| = o\left(sb(s) \right)$ of them have an endpoint in $S'_2$, there are $\Omega\left(sb(s)\right)$ of them with an endpoint in $V(L'_1) \setminus S$. In particular, in both cases
\[
\left|N^5_{Q_2}(S) \cap V(L'_1)\right| = \Omega\left(sb(s)^2  d^{-2}\right).
\]
\end{proof}

In a supercritical random graph $G(d+1,p)$ with $p=\frac{1+\epsilon}{d}$ the giant component itself will likely not be an $\alpha$-expander for any constant $\alpha>0$. Indeed, let the \emph{$2$-core} of a graph $G$ be the maximal subgraph of $G$ of minimum degree at least two. Standard results imply that it is likely that there are logarithmically sized pendant trees attached to the $2$-core of $G(d+1,p)$, and also that that the $2$-core of $G(d+1,p)$ typically contains logarithmically long bare paths, both of which lead to logarithmically large sets whose neighbourhoods have constant size.

However, the results of Benjamini, Kozma and Wormald \cite{BKW14} and Krivelevich \cite{K18} imply that whp the giant component contains a linear sized subgraph which is an $\alpha$-expander. In the case of $Q^d_p$ is it a simple consequence of Theorem \ref{t:expansion}, using some ideas of Krivelevich \cite{K19}, that we can also pass to a linear sized subset of the giant component with a significantly better expansion ratio than that guaranteed for the whole of $L_1$ by Theorem \ref{t:expansion}.

\begin{proof}[Proof of Theorem \ref{t:expansionsimple2}]
Let $\alpha' \ll \epsilon$ and let us assume that the conclusion of Theorem \ref{t:expansion} and \eqref{e:bdrylargeset} holds with this $\alpha'$ for some constant $\beta'$. Note that, by Theorem \ref{t:AKSfine} there is some constant $\gamma = \gamma(\epsilon)$ such that whp $|V(L_1)| = (1+o(1))\gamma n$. By \eqref{e:bdrylargeset} whp every set $S \subseteq V(L_1)$ such that ${\alpha' n \leq |S| \leq (1-2\alpha')|V(L_1)|}$ is such that
\[
|N_{L_1}(S)| \geq \frac{\alpha' \beta'}{1-2\alpha'} |S| d^{-2} (\log d)^{-1} =:t|S|.
\]

Let $U \subseteq V(L_1)$ be a set of maximal cardinality such $|U| < \alpha' n$ and $|N_{L_1}(U)| < t|X|$ and let $H := L_1 -X$. Note that $|V(H)| \geq (1- \frac{2\alpha'}{\gamma})|V(L_1)|$. Suppose that there is some subset ${W \subseteq V(H)}$ with $|W| \leq \frac{|V(H)|}{2}$ and $|N_H(W)| < t|W|$. Then, in particular, \[
|N_{L_1}(U \cup W)| \leq |N_{L_1}(U)| + |N_H(W)| \leq  t|U \cup W|.\]

By construction, ${|U \cup W| \leq \alpha' n + \frac{|V(H)|}{2} \leq (1-2\alpha') |V(L_1)|}$, since $|V(L_1)|= (1+o(1))\gamma n$ and $\alpha' \ll \epsilon$. However, by maximality of $U$ it follows that $|U \cup W| \geq \alpha' n$. Hence $U \cup W$ contradicts our assumption on the expansion of $L_1$.

It follows that $H$ is a $t$-expander, and so the conclusion of the theorem follows with $\alpha = \frac{2\alpha'}{\gamma}$ and ${\beta = \frac{\alpha' \beta'}{1-2\alpha'}}$.
\end{proof}

\section{Consequences of expansion in the giant component}\label{s:consequences}
\subsection{Mixing time of the lazy random walk}\label{s:mixing} Given a graph ${G=(V,E)}$, the \emph{lazy simple random walk} on $G$ is a random walk on $V$ which remains at the same vertex with probability $\frac{1}{2}$ in each time step, and otherwise moves to a uniformly chosen random neighbour of its current position. The \emph{stationary distribution} $\pi$ of the lazy random walk is given by $\pi(v) = \frac{d_G(v)}{2|E|}$ for each $v \in V$. For a subset $S \subseteq V$ let us write $\pi(S) = \sum_{v\in S} \pi(v)$ and let
\[
\pi_{\textrm{min}} = \min \{ \pi(v) \colon v \in V \}.
\]
 The \emph{edge measure} $Q$ of the random walk is given by
\[
Q(x,y) := \pi(x) P(x,y) \qquad \text{ and } \qquad Q(A,B) = \sum_{x \in A, y \in B} Q(x,y),
\]
where $P$ is the transition matrix of the lazy random walk, so that
\[
Q(x,y)= \begin{cases}
\frac{1}{4|E|}  &\mbox{if } xy \in E;\\
0  &\mbox{otherwise }. 
\end{cases}
\]
The \emph{bottleneck ratio} of a subset $S \subseteq V(G)$ is defined to be 
\[
\Phi(S) = \frac{Q(S,S^c)}{\pi(S)} = \frac{e_G(S,S^c)}{2d_G(S)},
\]
where $d_G(S) = \sum_{v \in S} d_G(v)$ is the \emph{total degree} of $S$. The bottleneck ratio of the random walk, sometimes known as the \emph{Cheeger constant} of $G$, is given by
\[
\Phi(G) := \min_{S \colon \pi(S) \leq \frac{1}{2}} \Phi(S).
\]
Note that, for a $k$-regular graph $G$, $\Phi(G) \geq \alpha$ is equivalent to $G$ being an $f(\alpha,k)$-edge-expander for some function $f(\alpha,k)$.

Let $P^t(v,\cdot)$ denote the distribution on $V$ given by starting the lazy random walk at $v \in V$ and running for $t$ steps, and let us define
\[
d(t) := \max_{v \in V} d_{TV}\left( P^t(v,\cdot), \pi \right)
\]
to be the maximal distance (over $v \in V$) between $P^t(v,\cdot)$ and the stationary distribution $\pi$, where we measure this distance in terms of the \emph{total variation distance}. That is, given two random variables $X$ and $Y$ distributed on the same finite set $Z$ we have
\[
d_{TV}(X,Y) = \frac{1}{2}\sum_{z\in Z} \big| \mathbb{P}(X=z) - \mathbb{P}(Y=z)\big|.
\]

The \emph{mixing time} of the lazy random walk is then defined as
\[
t_{\textrm{mix}} := \min \left\{ t \colon d(t) \leq \frac{1}{4} \right\}.
\]
See \cite{LPW17} for more background on mixing time of Markov chains.

It is relatively easy to use our result on the vertex-expansion of the giant component $L_1$ of $Q^d_p$ to show that the likely value of the Cheeger constant of $L_1$ is $\Omega\left(d^{-6}\right)$. Again, with a more careful argument we can improve this result somewhat.

\begin{lemma}\label{l:Cheeger}
Let $\epsilon >0$, let $p=\frac{1+\epsilon}{d}$ and let $L_1$ be the largest component in $Q^d_p$. Then whp $\Phi(L_1) = \Omega\left(d^{-5} \right) $.
\end{lemma}
\begin{proof}
Let $\alpha \ll \epsilon$ be a positive constant and let $\gamma$ be the survival probability of the ${\text{Po}(1+\epsilon)}$ branching process. Note that, since $L_1$ is connected, by Theorem \ref{t:AKSfine} whp \[
d_{L_1}(V(L_1)) \geq 2 (|V(L_1)|-1) \geq (2+o(1))\gamma n.
\] 
Let us further assume that the conclusions of Theorem \ref{t:expansion} and \eqref{e:bdrylargeset} hold with this $\alpha$ for some constant $\beta$. Note that, for every subset $S \subseteq V(L_1)$ we have that
\[
e_{Q^d_p}(S) = e_{L_1}(S) \qquad \text{ and } \qquad e_{Q^d_p}(S,S^c) = e_{L_1}(S,S^c).
\]

We first note that it is unlikely that any very large subset of $V(L_1)$ has small total degree. Indeed, suppose $S \subseteq V(L_1)$ is such that $|S| \geq (1-2\alpha) |V(L_1)|$ and $d_{L_1}(S) \leq \frac{d_{L_1}(V(L_1))}{2}$. Then it follows that $S^c$ is such that $|S^c| \leq 2\alpha |V(L_1)|\leq 2\alpha n$ and $d_{L_1}(S^c) \geq \frac{d_{L_1}(V(L_1))}{2} \geq (1+o(1)) \gamma n$.

However, $d_{Q^d}(X) = d|X| = e_{Q^d}(X,X^c) + 2e_{Q^d}(X)$ for any $X \subseteq V(Q^d)$. Hence, if $|X| \leq 2\alpha n$, then $\max \{ e_{Q^d}(X,X^c), 2e_{Q^d}(X) \} \leq 2\alpha d n$ and so $e_{Q^d_p}(X,X^c)$ and $e_{Q^d_p}(X)$ are both stochastically dominated by Bin$\left(2\alpha d n,p\right)$. Therefore, it follows by Lemma \ref{l:Chernoff} \ref{i:chernoff2} that
\[
\mathbb{P}\left( d_{Q^d_p}(X) \geq \frac{\gamma}{2} n \right) \leq 2 \left(\frac{4e\alpha(1+\epsilon)}{\gamma} \right)^{\frac{\gamma n}{2}}.
\]
Hence, by the union bound, the probability that there exists any set $S$ as above is at most 
\[
2\binom{n}{2\alpha n} \left(\frac{4e\alpha(1+\epsilon)}{\gamma} \right)^{\frac{\gamma n}{2}} \leq 2\left( \frac{e}{2\alpha}\right)^{\alpha n} \left(\frac{4e\alpha(1+\epsilon)}{\gamma} \right)^{\frac{\gamma n}{2}} = o(1),
\]
since $\alpha \ll \epsilon$.

In particular, whp
\begin{align*}
\Phi(L_1) &= \min \left\{ \Phi(S) \colon S \subseteq V(L_1) \text{ and }  d_{L_1}(S) \leq \frac{d_{L_1}(V(L_1))}{2}\right\} \\
&\geq \min \left\{ \Phi(S) \colon S \subseteq V(L_1) \text{ and } |S| \leq(1-2\alpha)n
\right\}.
\end{align*}

Next, we note that it is sufficient to consider the bottleneck ratio of connected subsets of $V(L_1)$. Indeed, let $S$ be an arbitrary subset of $V(L_1)$ with $|S| \leq (1-2\alpha) |V(L_1)|$ and let $C_1,C_2, \ldots, C_t$ be the connected components of $L_1[S]$. Then, for each $i$, $e_{L_1}(C_i, C_i^c) = e_{L_1}(C_i, S^c)$ and hence
\[
e_{L_1}(S,S^c) = \sum_{i=1}^t e_{L_1}(C_i, C_i^c) \qquad \text{ and } \qquad d_{L_1}(S) = \sum_{i=1}^t d_{L_1}(C_i).
\]
It follows that
\[
\Phi(S) = \frac{e_{L_1}(S,S^c)}{2d_{L_1}(S)} = \frac{\sum_{i=1}^t e_{L_1}(C_i, C_i^c)}{2\sum_{i=1}^t d_{L_1}(C_i)} \geq \min_i \left\{ \frac{e_{L_1}(C_i, C_i^c)}{ 2d_{L_1}(C_i)} \right\} = \min_i \{ \Phi(C_i) \}.
\]

Finally, let us bound the bottleneck ratio of connected subsets of $V(L_1)$. We note that, by Lemma \ref{l:excess}, there is a constant $C$ such that whp every subset $S \subseteq V(Q^d)$ which is connected in $Q^d_p$ satisfies $|S| \leq d$, or $e_{Q^d_p}(S) \leq C|S|$. Let $S$ be a connected subset of $V(L_1)$ with ${|S| \leq (1-2\alpha) |V(L_1)|}$.

Suppose first that $|S| \leq d$. Then, since $L_1$ is connected, $e_{Q^d_p}(S,S^c)\geq 1$ and hence
\[
\Phi(S) = \frac{e_{L_1}(S,S^c)}{2d_{L_1}(S)} \geq \frac{1}{2d|S|} = \Omega\left(d^{-2}\right).
\]

Suppose then that $d \leq |S| \leq (1-2\alpha)|V(L_1)|$, but $d_{Q^d_p}(S) \leq 4C|S|$. By Theorem \ref{t:expansion} and \eqref{e:bdrylargeset}, we have that $e_{Q^d_p}(S,S^c)\geq \beta  |S|d^{-5}$ and hence
\[
\Phi(S) = \frac{e_{L_1}(S,S^c)}{2d_{L_1}(S)} \geq \frac{ \beta  |S|d^{-5}}{8C |S|} = \Omega\left(d^{-5}\right).
\]

Finally, if $d \leq |S| \leq (1-2\alpha)|V(L_1)|$, and hence $e_{Q^d_p}(S) \leq C|S|$, and $d_{Q^d_p}(S) \geq 4C|S|$, then it follows that 
\[
e_{Q^d_p}(S,S^c) =d_{Q^d_p}(S) - 2e_{Q^d_p}(S) \geq \frac{ d_{Q^d_p}(S)}{2}
\]
and hence
\[
\Phi(S) = \frac{e_{L_1}(S,S^c)}{2d_{L_1}(S)} \geq \frac{1}{4}.
\]

Hence, we can conclude that 
\[
\Phi(L_1) = \min \{ \Phi(S) \colon |S| \leq (1-2\alpha)|V(L_1) \text{ and } S \text{ connected} \} = \Omega \left( d^{-5} \right).
\]
\end{proof}

We can relate the mixing time of the lazy random walk on $L_1$ to its Cheeger constant using the following theorem from Levin, Peres and Wilmer \cite{LPW17}. This theorem is a consequence of an important theorem of Sinclair and Jerrum \cite{SJ89} and of Lawler and Sokal \cite{LS88}, which bounds the Cheeger constant in terms of the spectral gap.

\begin{theorem}[{\cite[Theorem 17.10]{LPW17}}]\label{t:mix}
The mixing time of a lazy random walk on a graph $G$ satisfies the inequality
\[
t_{\textrm{mix}} \leq \frac{2}{\Phi(G)^2} \log \left( \frac{4}{\pi_{\textrm{min}}} \right).
\]
\end{theorem}

Theorem \ref{t:mixingtime} is then an immediate consequence of Lemma \ref{l:Cheeger} and Theorem \ref{t:mix}

\begin{proof}[Proof of Theorem \ref{t:mixingtime}]
We note that whp $|E(Q^d_p)| \leq n$ and so $\pi_{\textrm{min}} \geq \frac{1}{2n}$. Hence, by Lemma \ref{l:Cheeger} and Theorem \ref{t:mix}, whp
\[
t_{\textrm{mix}} \leq \frac{2}{\Phi(L_1)^2} \log \left( \frac{4}{\pi_{\textrm{min}}} \right) = O\left( d^{11} \right).
\]
\end{proof}

\subsection{Diameter}\label{s:diameter}
It is immediate from Theorem \ref{t:expansion} that whp the giant component of $Q^d_p$ in the supercritical regime has diameter $O\left(d^6\right)$. However, with a more careful argument we can improve this crude estimate.

\begin{proof}[Proof of Theorem \ref{t:diameter}]
We argue as in the proof of Theorem \ref{t:expansion}, using the same terminology. In particular we take $q_1,q_2, L'_1$, $L_1$, $b(s)$, $\mathcal{C}(s)$ and constants $\beta,c_1,c_2,\ldots$ as in the proof.

Note, by Lemma \ref{l:treedecomp}, for each $1 \leq s \leq \frac{|V(L_1)|}{2}$ each piece in $\mathcal{C}(s)$ has diameter at most ${r(s) :=2 c_8^{-1}b(s)^{-1}d}$. Also, as before it follows from Lemma \ref{l:outsidegiant} that 
\begin{equation}\label{e:restcomp}
\text{whp every component of $R:=L_1 - L'_1$ has order at most $c_4^{-1}d$.} 
\end{equation}

Furthermore, by Claim \ref{c:pathsmall}, whp for any $1 \leq t \leq 2s \leq |V(L_1)|$ and any partition of $\mathcal{C}(s)$ into two sets $\{\mathcal{C}_A, \mathcal{C}_B\}$, where $A := \bigcup \mathcal{C}_A$ and  $B := \bigcup \mathcal{C}_B$, with $\min \{|A|,|B| \} = t$ there is a family of at least $c_7 b(t) t$ vertex-disjoint $A\textrm{-}B$-paths of length at most five in $Q^d_{q_2}$. Let us assume that both of these likely events hold.

Let $v$ be an arbitrary vertex in $L'_1$ and let $S(0) = \{v\}$. We recursively define a sequence of vertex sets $S(i)$ as follows: given $S(i)$, let $t_i:= |S(i)|$ and let $S'(i)$ be the union of all pieces of $\mathcal{C}(t_i)$ which contain vertices in $S(i)$. If $t'_i:= |S'(i)| \geq \min \left\{ \frac{|V(L'_1)|}{2}, 2t_i \right\}$, then we let ${S(i+1) = S'(i)}$, otherwise by the above assumption there is a family of at least $c_7 t'_i b(t'_i)$ paths of length at most five in $Q^d_{q_2}$ between $S'(i)$ and its complement in $V(L'_1)$. In this case, we let $S(i+1)$ be $S'(i)$ together with the endpoints of these paths. 

We note that, since the diameter of each piece of $\mathcal{C}(t_i)$ is at most $r(t_i)$, we have that 
\begin{equation}\label{e:increment}
S(i+1) \subseteq N^{r(t_i)+5}_{Q_2}(S_i).
\end{equation}

Since $f(t) = t b(t)$ is an increasing function of $t$ for $t \leq \frac{n}{2}$, it follows that, for each $i\geq 0$ with $t_i \leq \frac{|V(L'_1)|}{2}$, we have 
\begin{equation}
t_{i+1} \geq \min \{ t'_i + c_7 t'_i b(t'_i) , 2t_i \} \geq t_i\left(1 + c_7b(t_i)\right).
\end{equation}

So, let us analyse the growth rate of the sequence $(x_i)$ defined recursively as $x_0=1$ and ${x_{i+1} = x_i\cdot\left(1 + c_7b(x_i)\right)}$. 

We claim that 
\[
\text{for any $0 < \epsilon \leq \frac{1}{2}$, if $2^{(1-2\epsilon)d}  \leq x_i \leq 2^{(1-\epsilon)d}$, then $x_{i+ 2c_7^{-1} d} \geq 2^{(1-\epsilon)d}$.}
\]
Indeed, for any $x_j \leq 2^{(1-\epsilon)d}$ we have that $b(x_j) \geq \epsilon$ and hence
\[
x_{i+2c_7^{-1} d} \geq \min \left\{ 2^{(1-\epsilon)d}, x_i\left(1 + c_7\epsilon \right)^{2c_7^{-1} d}\right\} \geq \min \left\{ 2^{(1-\epsilon)d}, 2^{(1-2\epsilon)d} e^{\epsilon d}\right\} \geq 2^{(1-\epsilon)d},
\]
using that $(1+\alpha) \geq e^{\frac{\alpha}{2}}$ for any $0 \leq \alpha \leq \frac{1}{2}$. 

It follows that for any $\epsilon >0$ there are at most $2c_7^{-1} d$ many $i$ such that $2^{(1-2\epsilon)d}  \leq t_i \leq 2^{(1-\epsilon)d}$. Note that, if $2^{(1-2\epsilon)d} \leq t_i \leq 2^{(1-\epsilon)d}$, then $r(t_i) \leq 2c_8^{-1} \epsilon^{-1} d$.

Hence, if we let $I_j =\left\{ i \colon 2^{(1-2^{-j})d} \leq t_i \leq 2^{(1-2^{-(j+1)})d} \right\}$ for each $j=0, 1, \ldots j_{\textrm{max}}$, where $j_{\textrm{max}}$ is minimal such that $2^{(1-2^{-(j_{\textrm{max}}+1)})d} > \frac{|V(L_1)|}{2}$, then $|I_j| \leq 2c_7^{-1} d$ for each $j$ and so
\begin{align*}
\sum_{i \colon t_i \leq \frac{|V(L_1)|}{2}} \left(r(t_i)+5\right) \leq \sum_{j=0}^{j_{\textrm{max}}} \sum_{i \in I_j} \left(r(t_i)+5\right) \leq \sum_{j=0}^{j_{\textrm{max}}} \sum_{i \in I_j} \left( 2c_8^{-1} 2^{j+1}d+5 \right)= O\left(d^2 \sum_{j=0}^{j_{\textrm{max}}}  2^{j+1} \right)= O\left(d^3 \right).
\end{align*}

In particular, by \eqref{e:increment} there is some constant $C$ such that for each $v \in V(L'_1)$
\[
\left|N^{Cd^3}_{Q_2}(v) \cap V(L'_1)\right| > \frac{|V(L'_1)|}{2},
\]
and so the distance between any two vertices in $L'_1$ in $Q_2$ is at most $2Cd^3$.

Finally, by \eqref{e:restcomp} every vertex in $L_1$ is at distance at most $c_4^{-1}d$ from a vertex in $L'_1$ and hence the diameter of $L_1$ is at most
\[
2Cd^3 + 2c_4^{-1}d = O\left(d^3 \right).
\]
\end{proof}

\subsection{Long cycles and large minors}\label{s:cycleminor}
\begin{proof}[Proof of Theorem \ref{t:cycles}]
Let $\alpha \ll \epsilon$ and let $L_1$ be the largest component of $Q^d_p$. Then by Theorem \ref{t:expansion} \ref{i:large} whp every $S \subseteq V(L_1)$ such that $\alpha n \leq |S| \leq \frac{|V(L_1)|}{2}$ satisfies ${\left|N_{Q^d_p}(S)\right|  \geq  \beta n d^{-2} (\log d)^{-1}}$.

Hence, applying Theorem \ref{t:cycleexpander} with $k = \frac{|V(L_1)|}{2}$ and $t = \beta n d^{-2} (\log d)^{-1}$, we can conclude that $L_1$ contains a cycle of length $\Omega\left( n d^{-2} (\log d)^{-1}\right)$.

\end{proof}

\begin{proof}[Proof of Theorem \ref{t:minors}]
Let $\alpha \ll \epsilon$ and let $L_1$ be the largest component of $Q^d_p$. Again by Theorem \ref{t:expansion} \ref{i:large} whp every $S \subseteq V(L_1)$ such that $\alpha n \leq |S| \leq \frac{|V(L_1)|}{2}$ satisfies $\left|N_{Q^d_p}(S)\right|  \geq  \beta n d^{-2} (\log d)^{-1}$.

 If $L_1$ does not contain a $K_t$-minor, then by Theorem \ref{t:minorexpander} there is some constant $C >0$ such that $V(L_1)$ contains a subset $X$ of size at most $C t \sqrt{|V(L_1)|} \leq Ct \sqrt{n}$, such that each component of $G - X$ has order at most $\frac{2|V(L_1)|}{3}$. It follows that there is some subset $S \subseteq V(L_1)$, which is the union of some components of $G-X$, such that $\frac{|V(L_1)|}{3}\leq |S| \leq\frac{|V(L_1)|}{2}$ and $N_{Q^d_p}(S) \subseteq X$, and so by Theorem \ref{t:expansion} \ref{i:large}
\[
\beta n d^{-2} (\log d)^{-1} \leq \left|N_{Q^d_p}(S)\right| \leq |X| \leq Ct \sqrt{n}.
\]
It follows that $t \geq\frac{\beta \sqrt{n}}{ Cd^2 \log d } = \Omega \left( \sqrt{n} d^{-2} (\log d)^{-1}\right)$.
\end{proof}

\section{Discussion}\label{s:discussion}

Theorem \ref{t:expansionsimple1} gives a good bound on the likely expansion of the giant component of $Q^d_p$ in the supercritical regime, although it is unlikely to be optimal in terms of its dependence on $d$. It would be interesting to determine the optimal expansion ratio.

However, perhaps this is not quite the right question to ask. As explained in the introduction, in the case of $G(d+1,p)$, whp the giant component itself is not an $\alpha$-expander for any constant $\alpha>0$, but whp contains a linear sized subgraph which is. Analogously, it seems likely that there should be a large subgraph of the giant component of $Q^d_p$ which will have a significantly better expansion ratio than that of the giant component itself.

\begin{question}
For what $\alpha = \alpha(d)$ is it true that whp $Q^d_p$ contains a subgraph of size $\Omega(n)$ which is an $\alpha$-expander?
\end{question}

As we saw in Theorem \ref{t:expansionsimple2}, we can take $\alpha(d) = d^{-2} (\log d)^{-1}$. We note that some inverse polynomial power of $d$ is still necessary in the expansion ratio here. 
\begin{claim}\label{c:notexpander}
Let $\delta > 0$, let $\alpha \in (0,1)$ and let $p = \frac{\delta}{d}$. Then there exists a $C=C(\delta,\alpha)$ such that whp there are no subsets $W \subseteq V(Q^d)$ of size at least $\alpha n$ such that $Q^d_p[W]$ is a $\frac{C}{d}$-expander.
\end{claim}
\begin{proof}
We first note that for any $i \in [d]$ there are $\frac{n}{2}$ edges `in direction $i$' in $Q^d$, that is, between two vertices which differ in the $i$th coordinate. Hence, it is a simple consequence of Lemma \ref{l:Chernoff} that there is some $C'=C'(\delta)>0$ such that whp for any $i \in [d]$ there are less than $\frac{C'n}{d}$ edges in direction $i$.

We will show that there is some $\beta >0$ such that for any subset $W \subseteq V(Q^d)$ of size at least $\alpha n$ there is some $i \in [d]$ such that 
\[
W_i^0 := \{ v \in W \colon v_i = 0\} \qquad \text{ and } \qquad W_i^1 := \{ v \in W \colon v_i = 1\}
\]
both have size at least $\alpha \beta n$. By our assumption on the number of edges in direction $i$, 
\[
\left|N\left(W_i^j\right)\right| < \frac{C'n}{d} < \frac{C'\left|W_i^j\right|}{\alpha\beta d}, \qquad \text{ for } j \in \{0,1\}. 
\]
However, since there is some $j \in \{0,1\}$ such that $\left|W^j_i\right| \leq \frac{|W|}{2}$, it follows that $Q^d_p[X]$ is not a $\frac{C'}{\alpha \beta d}$-expander, and so the claim holds with $C =\frac{C'}{\alpha \beta}$.

A simple way to show the existence of $\beta$ is using the language of discrete entropy. Let $X$ be a uniformly random chosen element of $W$ and let $X_i$ be the projection of $X$ to the $i$th coordinate. Note that $X_i \sim \text{Ber}(p_i)$ where $p_i$ is the proportion of $v \in W$ with $v_i =1$, i.e., $p_i = \frac{|W_i^1|}{|W|}$.

Then, by Lemma \ref{l:entropy} \ref{i:unifent}, $H(X) = \log_2 (|W|) \geq d + \log \alpha \geq \frac{d}{2}$ and, by Lemma \ref{l:entropy} \ref{i:jointent},
\[
H(X) \leq \sum_{i=1}^d H(X_i) = \sum_{i=1}^d h(p_i) \leq d \max_i h(p_i)
\]
where $h(x) = - x \log_2 x - (1-x)\log_2 (1-x)$ is the binary entropy function. In particular, it follows that ${\max_i h(p_i) \geq \frac{1}{2}}$. However, since $h$ is symmetric around $\frac{1}{2}$ and increasing on $[0,\frac{1}{2}]$, we can conclude that there is some $\delta$ such that $\min \{p_i,1-p_i\} \geq \beta$. Hence
\[
\left|W_i^j\right| = |W| \cdot \mathbb{P}(X_i = j) \geq \beta|W| \geq \alpha \beta n, \qquad \text{ for } j \in \{0,1\},
\]
as claimed.
\end{proof}
As consequences of the expansion properties of the giant component, we deduced bounds on its mixing time, diameter, circumference and Hadwiger number which are almost optimal, up to some polynomial factors in $d$. However, it seems unlikely that any of these results are optimal in terms of their dependence in $d$.

For the diameter and the mixing time it is not clear what the `correct' answer should be. It seems likely, but it is not immediate, that $Q^d_p$ should have larger diameter and mixing time than $Q^d$ does. For the diameter it might be that $O(d)$ is the correct order of growth. However, whilst the mixing time of the lazy random walk on $Q^d$ is known to be $O(d \log d)$, see, e.g., \cite{LPW17}, it can be shown that the largest component $L_1$ of $Q^d_p$ whp contains bare paths of length $\Omega(d)$. In particular, since we expect a lazy random walk starting in the middle of such a path to take at least $\Omega(d^2)$ steps before reaching either endpoint, it follows that whp the mixing time of the lazy random walk on $L_1$ is $\Omega\left(d^2\right)$.

In the case of $G(d+1,p)$, Benjamini, Kozma and Wormald \cite{BKW14} used their description of the structure of the giant component as a decorated expander to give a $\Theta\left((\log d)^2\right)$ bound on the mixing time of the lazy random walk on the giant component in the supercritical regime. Roughly, the idea here is that the lazy random walk mixes quickly, in time $O(\log d)$, inside the expanding subgraph, but it might end up making detours of length $\Omega\left((\log d)^2\right)$ inside the decorations. If a similar description of the giant component of $Q^d_p$ were to hold, then we might hope that the expanding subgraph has expansion ratio $\alpha$ which is a small inverse power of $d$, and these decorations have size $O(d)$, see Lemma \ref{l:outsidegiant}, in which case perhaps a more reasonable hope would be that the mixing time of the lazy random on the giant component is either dominated by the mixing time on the $\alpha$-expanding subgraph, which would be of order something like $\alpha^{-2}$, or by the length of the detours in the decorations, of order something like $d^2$. Note that, by Claim \ref{c:notexpander}, both of these terms must be $\Omega(d^2)$.

\begin{question}
Let $\epsilon >0$, let $p=\frac{1+\epsilon}{d}$ and let $L_1$ be the largest component of $Q^d_p$.
\begin{itemize}
\item How large is the likely diameter of $L_1$?
\item What is the likely mixing time of the lazy random walk on $L_1$?
\end{itemize}
\end{question}

In the context of Bollob\'{a}s, Kohayakawa and {\L}uczak's \cite{BKL94d} question about when the diameter of the largest component of $Q^d_p$ is superpolynomial in $d$, it would be interesting to know if our results can be extended partially into the weakly supercritical regime, for example when $\epsilon=o(1)$ is significantly larger than $d^{-1}$ which was the regime considered in \cite{BKL92}. 

In terms of the circumference and the Hadwiger number of $Q^d_p$ there are more natural conjectures to make, analogous to the case of $G(d+1,p)$, which is that whp $Q^d_p$ contains a cycle whose length is linear in $n$ and a complete minor of size $\Omega(\sqrt{n})$. However it seems unlikely that it is possible to prove such sharp results simply by considering the expansion properties of the giant component. In particular, we note that it may be the case that it is easier to show the likely existence of a linear length path, than that of a cycle.

\begin{question}Let $\epsilon>0$ and $p=\frac{1+\epsilon}{d}$.
\begin{enumerate}
\item Is it the case that whp $Q^d_p$ contains a path of length $\Omega(n)$?
\item Is it the case that whp $Q^d_p$ contains a cycle of length $\Omega(n)$?
\item Is it the case that whp $Q^d_p$ contains a complete minor of order $\Omega\left(\sqrt{n}\right)$?
\end{enumerate}
\end{question}

Furthermore, in the case of a positive answer it would also be interesting to know the dependence of the leading constants on $\epsilon$. For example, in $G(d+1,p)$ it is known that for $p = \frac{1+\epsilon}{d}$ the giant component is of order $(2\epsilon + o(\epsilon))d$, the length of the longest cycle is of order $\Theta\left(\epsilon^2\right) d$ (see, for example, \cite[Theorem 5.7]{JLR00}) and the size of the largest complete minor is of order $\Theta\left(\epsilon^\frac{3}{2}\right) \sqrt{d}$ (see \cite{FKO09}).

There are also some interesting open questions about the model $Q^d_p$ in the paper of Condon, Espuny D{\'\i}az, Gir{\~a}o, K{\"u}hn and Osthus \cite{CDGKO21}. In particular, they used as a crucial part of their proof the fact that whp $Q^d_{\frac{1}{2}}$ contains an `almost spanning' path, that is, a path containing $(1-o(1))n$ vertices, and they showed that this property is in fact true for $Q^d_p$ for any constant $p$. 

However, analogous to the case of $G(d+1,p)$, we should perhaps expect such a path to exist for much smaller values of $p$. In particular, if we expect the sparse random subgraph $Q^d_{p}$ with $p=\frac{c}{d}$ to contain a path of linear length $f(c)n$ for some function $f(c)$ when $c > 1$, it is natural to conjecture that $f(c) \to 1$ as $c \rightarrow \infty$.

\begin{question}[\cite{CDGKO21}]
Let $p = \omega\left(\frac{1}{d}\right)$. Is it true that whp $Q^d_p$ contains a path of length $(1-o(1))n$?
\end{question}

Finally, other notions of random subgraphs of the hypercube $Q^d$ have also been studied. In particular, if we let $Q^d(p)$ denote a random \emph{induced} subgraph of $Q^d$, obtained by retaining each \emph{vertex} independently with probability $p$, then the typical existence of a giant component in $Q^d(p)$ when $p=\frac{1+\epsilon}{d}$ for a fixed $\epsilon >0$ was shown by Bollob\'{a}s, Kohayakawa and {\L}uczak \cite{BKL94}, and this was extended to a broader range of $p$ with $\epsilon=o(1)$ by Reidys \cite{R09}. It would be interesting to know if whp the giant component in $Q^d(p)$ also has good expansion properties. We note that random induced subgraphs of pseudo-random $d$-regular graphs have been studied by Diskin and Krivelevich \cite{DK21}, who in particular prove some likely expansion properties of the giant component in a supercritical random induced subgraph.

\section*{Acknowledgement}
The third author wishes to thank Asaf Nachmias for stimulating discussions and helpful remarks.

	\bibliographystyle{plain}
	\bibliography{cube}

\begin{thebibliography}{10}

\bibitem{AKS81a}
M.~{Ajtai}, J.~{Koml\'{o}s}, and E.~{Szemer\'{e}di}.
\newblock {The longest path in a random graph}.
\newblock {\em {Combinatorica}}, 1:1--12, 1981.

\bibitem{AKS81}
M.~Ajtai, J.~Koml\'{o}s, and E.~Szemer\'{e}di.
\newblock Largest random component of a {$k$}-cube.
\newblock {\em Combinatorica}, 2(1):1--7, 1982.

\bibitem{AS16}
N.~{Alon} and J.~H. {Spencer}.
\newblock {\em {The probabilistic method}}.
\newblock Hoboken, NJ: John Wiley \& Sons, fourth edition, 2016.

\bibitem{BKW14}
I.~Benjamini, G.~Kozma, and N.~Wormald.
\newblock The mixing time of the giant component of a random graph.
\newblock {\em Random Structures Algorithms}, 45(3):383--407, 2014.

\bibitem{BLPS18}
N.~{Berestycki}, E.~{Lubetzky}, Y.~{Peres}, and A.~{Sly}.
\newblock {Random walks on the random graph}.
\newblock {\em {Ann. Probab.}}, 46(1):456--490, 2018.

\bibitem{B67}
A.~J. Bernstein.
\newblock Maximally connected arrays on the $n$-cube.
\newblock {\em SIAM J. Appl. Math.}, 15:1485--1489, 1967.

\bibitem{BFM98}
A.~{Beveridge}, A.~{Frieze}, and C.~{McDiarmid}.
\newblock {Random minimum length spanning trees in regular graphs}.
\newblock {\em {Combinatorica}}, 18(3):311--333, 1998.

\bibitem{B83}
B.~Bollob\'{a}s.
\newblock The evolution of the cube.
\newblock In {\em Combinatorial mathematics ({M}arseille-{L}uminy, 1981)},
  volume~75 of {\em North-Holland Math. Stud.}, pages 91--97. North-Holland,
  Amsterdam, 1983.

\bibitem{B84}
B.~{Bollob\'as}.
\newblock {The evolution of random graphs}.
\newblock {\em {Trans. Am. Math. Soc.}}, 286:257--274, 1984.

\bibitem{B90}
B.~Bollob\'{a}s.
\newblock Complete matchings in random subgraphs of the cube.
\newblock {\em Random Structures Algorithms}, 1(1):95--104, 1990.

\bibitem{B01}
B.~Bollob\'{a}s.
\newblock {\em Random graphs}.
\newblock Cambridge Studies in Advanced Mathematics. Cambridge University
  Press, second edition, 2001.

\bibitem{BKL92}
B.~Bollob\'{a}s, Y.~Kohayakawa, and T.~{\L}uczak.
\newblock The evolution of random subgraphs of the cube.
\newblock {\em Random Structures Algorithms}, 3(1):55--90, 1992.

\bibitem{BKL94d}
B.~Bollob\'{a}s, Y.~Kohayakawa, and T.~{\L}uczak.
\newblock On the diameter and radius of random subgraphs of the cube.
\newblock {\em Random Structures Algorithms}, 5(5):627--648, 1994.

\bibitem{BKL94}
B.~Bollob\'{a}s, Y.~Kohayakawa, and T.~{\L}uczak.
\newblock On the evolution of random {B}oolean functions.
\newblock In {\em Extremal problems for finite sets ({V}isegr\'{a}d, 1991)},
  volume~3 of {\em Bolyai Soc. Math. Stud.}, pages 137--156. J\'{a}nos Bolyai
  Math. Soc., Budapest, 1994.

\bibitem{BR06}
B.~{Bollob\'as} and O.~{Riordan}.
\newblock {\em {Percolation.}}
\newblock Cambridge: Cambridge University Press, 2006.

\bibitem{BCVSS06}
C.~Borgs, J.~T. Chayes, R.~van~der Hofstad, G.~Slade, and J.~Spencer.
\newblock Random subgraphs of finite graphs. {III}. {T}he phase transition for
  the {$n$}-cube.
\newblock {\em Combinatorica}, 26(4):395--410, 2006.

\bibitem{BH57}
S.~B. {Broadbent} and J.~M. {Hammersley}.
\newblock {Percolation processes. I: Crystals and mazes}.
\newblock {\em {Proc. Camb. Philos. Soc.}}, 53:629--641, 1957.

\bibitem{CL01}
F.~{Chung} and L.~{Lu}.
\newblock {The diameter of sparse random graphs}.
\newblock {\em {Adv. Appl. Math.}}, 26(4):257--279, 2001.

\bibitem{CDGKO21}
P.~Condon, A.~Espuny~D{\'\i}az, A.~Gir{\~a}o, D.~K{\"u}hn, and D.~Osthus.
\newblock Hamiltonicity of random subgraphs of the hypercube.
\newblock {\em arXiv preprint arXiv:2007.02891}, 2020.

\bibitem{DLP14}
J.~{Ding}, E.~{Lubetzky}, and Y.~{Peres}.
\newblock {Anatomy of the giant component: the strictly supercritical regime}.
\newblock {\em {Eur. J. Comb.}}, 35:155--168, 2014.

\bibitem{DK21}
S.~Diskin and M.~Krivelevich.
\newblock Site percolation on pseudo-random graphs.
\newblock {\em arXiv preprint arXiv:2107.13326}, 2021.

\bibitem{EJ18}
S.~Ehard and F.~Joos.
\newblock Paths and cycles in random subgraphs of graphs with large minimum
  degree.
\newblock {\em Electron. J. Combin.}, 25(2):Paper No. 2.31, 2018.

\bibitem{EKK20}
J.~Erde, M.~Kang, and M.~Krivelevich.
\newblock Large complete minors in random subgraphs.
\newblock {\em Combin. Probab. Comput.}, 30(4):619--630, 2021.

\bibitem{ER59}
P.~{Erd\H{o}s} and A.{R\'enyi}.
\newblock {On random graphs. I}.
\newblock {\em {Publ. Math.}}, 6:290--297, 1959.

\bibitem{ES79}
P.~Erd\H{o}s and J.~Spencer.
\newblock Evolution of the {$n$}-cube.
\newblock {\em Comput. Math. Appl.}, 5(1):33--39, 1979.

\bibitem{FR07}
D.~{Fernholz} and V.~{Ramachandran}.
\newblock {The diameter of sparse random graphs}.
\newblock {\em {Random Structures Algorithms}}, 31(4):482--516, 2007.

\bibitem{FKO08}
N.~Fountoulakis, D.~K\"{u}hn, and D.~Osthus.
\newblock The order of the largest complete minor in a random graph.
\newblock {\em Random Structures Algorithms}, 33(2):127--141, 2008.

\bibitem{FKO09}
N.~Fountoulakis, D.~K\"{u}hn, and D.~Osthus.
\newblock Minors in random regular graphs.
\newblock {\em Random Structures Algorithms}, 35(4):444--463, 2009.

\bibitem{FR08}
N.~{Fountoulakis} and B.~A. {Reed}.
\newblock {The evolution of the mixing rate of a simple random walk on the
  giant component of a random graph}.
\newblock {\em {Random Structures Algorithms}}, 33(1):68--86, 2008.

\bibitem{FK16}
A.~Frieze and M.~Karo\'{n}ski.
\newblock {\em Introduction to random graphs}.
\newblock Cambridge University Press, Cambridge, 2016.

\bibitem{FK13}
A.~Frieze and M.~Krivelevich.
\newblock On the non-planarity of a random subgraph.
\newblock {\em Combin. Probab. Comput.}, 22(5):722--732, 2013.

\bibitem{G59}
E.~N. Gilbert.
\newblock Random graphs.
\newblock {\em Ann. Math. Statist.}, 30:1141--1144, 1959.

\bibitem{G89}
G.~{Grimmett}.
\newblock {\em {Percolation.}}
\newblock Berlin: Springer, 1999.

\bibitem{H64}
L.~H. Harper.
\newblock Optimal assigments of numbers to vertices.
\newblock {\em SIAM J. Appl. Math.}, 12:131--135, 1964.

\bibitem{H76}
S.~Hart.
\newblock A note on the edges of the $n$-cube.
\newblock {\em Discrete Math.}, 14:157--163, 1976.

\bibitem{HHKLLW21}
J.~Haslegrave, J.~Hu, J.~Kim, H.~Liu, B.~Luan, and G.~Wang.
\newblock Crux and long cycles in graphs.
\newblock {\em arXiv preprint arXiv:2107.02061}, 2021.

\bibitem{HH11}
M.~Heydenreich and R.~van~der Hofstad.
\newblock Random graph asymptotics on high-dimensional tori {II}: {V}olume,
  diameter and mixing time.
\newblock {\em Probab. Theory Related Fields}, 149(3-4):397--415, 2011.

\bibitem{HN14}
R.~van~der Hofstad and A.~Nachmias.
\newblock Unlacing hypercube percolation: {A} survey.
\newblock {\em Metrika}, 77(1):23--50, 2014.

\bibitem{HN17}
R.~van~der Hofstad and A.~Nachmias.
\newblock Hypercube percolation.
\newblock {\em J. Eur. Math. Soc.}, 19(3):725--814, 2017.

\bibitem{HLW06}
S.~Hoory, N.~Linial, and A.~Wigderson.
\newblock Expander graphs and their applications.
\newblock {\em Bull. Amer. Math. Soc. (N.S.)}, 43(4):439--561, 2006.

\bibitem{HN20}
T.~Hulshof and A.~Nachmias.
\newblock Slightly subcritical hypercube percolation.
\newblock {\em Random Structures Algorithms}, 56(2):557--593, 2020.

\bibitem{JLR00}
S.~Janson, T.~{\L}uczak, and A.~Ruci\'{n}ski.
\newblock {\em Random graphs}.
\newblock Wiley-Interscience Series in Discrete Mathematics and Optimization.
  Wiley-Interscience, 2000.

\bibitem{KR10}
K.~Kawarabayashi and B.~Reed.
\newblock A separator theorem in minor-closed classes.
\newblock In {\em 2010 {IEEE} 51st {A}nnual {S}ymposium on {F}oundations of
  {C}omputer {S}cience---{FOCS} 2010}, pages 153--162. IEEE Computer Soc., Los
  Alamitos, CA, 2010.

\bibitem{K82}
H.~{Kesten}.
\newblock {\em {Percolation theory for mathematicians}}.
\newblock Birkh\"auser, Boston, MA, 1982.

\bibitem{KSS80}
J.~Koml\'{o}s, M.~Sulyok, and E.~Szemer\'{e}di.
\newblock Second largest component in a random graph.
\newblock {\em Studia Sci. Math. Hungar.}, 15(4):391--395, 1980.

\bibitem{K18}
M.~Krivelevich.
\newblock Finding and using expanders in locally sparse graphs.
\newblock {\em SIAM J. Discrete Math.}, 32(1):611--623, 2018.

\bibitem{K19}
M.~Krivelevich.
\newblock Expanders -- how to find them, and what to find in them.
\newblock In {\em Surveys in combinatorics 2019}, volume 456 of {\em London
  Math. Soc. Lecture Note Ser.}, pages 115--142. Cambridge Univ. Press,
  Cambridge, 2019.

\bibitem{K19a}
M.~Krivelevich.
\newblock Long cycles in locally expanding graphs, with applications.
\newblock {\em Combinatorica}, 39(1):135--151, 2019.

\bibitem{KN06}
M.~Krivelevich and A.~Nachmias.
\newblock Coloring complete bipartite graphs from random lists.
\newblock {\em Random Structures Algorithms}, 29(4):436--449, 2006.

\bibitem{KS14}
M.~Krivelevich and W.~Samotij.
\newblock Long paths and cycles in random subgraphs of {$\mathcal{H}$}-free
  graphs.
\newblock {\em Electron. J. Combin.}, 21(1):Paper 1.30, 2014.

\bibitem{KS13}
M.~Krivelevich and B.~Sudakov.
\newblock The phase transition in random graphs: {A} simple proof.
\newblock {\em Random Structures Algorithms}, 43(2):131--138, 2013.

\bibitem{LS88}
G.~F. {Lawler} and A.~D. {Sokal}.
\newblock {Bounds on the \(L^ 2\) spectrum for Markov chains and Markov
  processes: A generalization of Cheeger's inequality}.
\newblock {\em {Trans. Am. Math. Soc.}}, 309(2):557--580, 1988.

\bibitem{LPW17}
D.~A. {Levin}, Y.~{Peres}, and E.~L. {Wilmer}.
\newblock {\em {Markov chains and mixing times}}.
\newblock Providence, RI: American Mathematical Society, 2017.

\bibitem{L64}
J.~H. Lindsey.
\newblock Assigment of numbers to vertices.
\newblock {\em Amer. Math. Monthly}, 71:508--516, 1964.

\bibitem{L90}
T.~{{\L}uczak}.
\newblock {Component behavior near the critical point of the random graph
  process}.
\newblock {\em {Random Structures Algorithms}}, 1(3):287--310, 1990.

\bibitem{MSW21}
C.~McDiarmid, A.~Scott, and P.~Withers.
\newblock The component structure of dense random subgraphs of the hypercube.
\newblock {\em {Random Structures Algorithms}}, 59(1):3--24, 2021.

\bibitem{P08}
G.~{Pete}.
\newblock {A note on percolation on \(\mathbb Z^d\): {I}soperimetric profile
  via exponential cluster repulsion}.
\newblock {\em {Electron. Commun. Probab.}}, 13:377--392, 2008.

\bibitem{R09}
C.~M. Reidys.
\newblock Large components in random induced subgraphs of {$n$}-cubes.
\newblock {\em Discrete Math.}, 309(10):3113--3124, 2009.

\bibitem{RW10}
O.~{Riordan} and N.~{Wormald}.
\newblock {The diameter of sparse random graphs}.
\newblock {\em {Comb. Probab. Comput.}}, 19(5-6):835--926, 2010.

\bibitem{S67}
A.~A. Sapo\v{z}enko.
\newblock Metric properties of almost all functions of the algebra of logic.
\newblock {\em Diskret. Analiz}, (10):91--119, 1967.

\bibitem{SJ89}
A.~{Sinclair} and M.~{Jerrum}.
\newblock {Approximate counting, uniform generation and rapidly mixing Markov
  chains}.
\newblock {\em {Inf. Comput.}}, 82(1):93--133, 1989.

\end{thebibliography}
	
\appendix

\end{document}